\documentclass[a4paper,11pt]{amsart}
\usepackage{amssymb}
\usepackage{amscd}
\usepackage{slashed}

\usepackage[backref]{hyperref}

\newtheorem{theorem}{Theorem}[section]
\newtheorem{lemma}[theorem]{Lemma}
  
\newtheorem{proposition}[theorem]{Proposition}
\newtheorem{corollary}[theorem]{Corollary}

\theoremstyle{definition}
\newtheorem{definition}[theorem]{Definition}

\theoremstyle{remark}
\newtheorem{remark}[theorem]{Remark}

\newtheorem{problem}[theorem]{Problem}

\usepackage{graphicx} 
\usepackage{amsmath} 
\usepackage{amsfonts}
\usepackage{amssymb}

\newcommand{\be}{\begin{equation}}

\newcommand{\ee}{\end{equation}}

\newcommand{\g}{g^o}

\newcommand{\om}{\omega}

\newcommand{\cG}{{\mathcal G}}
\newcommand{\cU}{\mathcal{U}}

\renewcommand{\H}{B}

\newcommand{\al}{\mbox{\boldmath$\Delta$}}

\newcommand{\aNd}{\mbox{\boldmath$ \nabla$}\hspace{-.6mm}}

\newcommand{\si}{\sigma}

\newcommand{\ba}{\begin{array}}

\newcommand{\ea}{\end{array}}

\newcommand{\beq}{\begin{eqnarray}}

\newcommand{\eeq}{\end{eqnarray}}

\newtheorem{lm}{lemma}

\newtheorem{thee}{theorem}

\newtheorem{proo}{proposition}

\newtheorem{co}{corollary}

\newtheorem{rem}{remark}

\newtheorem{deff}{definition}

\newcommand{\bd}{\begin{deff}}

\newcommand{\ed}{\end{deff}}

\newcommand{\bl}{\begin{lm}}

\newcommand{\el}{\end{lm}}

\newcommand{\bp}{\begin{proo}}

\newcommand{\ep}{\end{proo}}

\newcommand{\bt}{\begin{thee}}

\newcommand{\et}{\end{thee}}

\newcommand{\bc}{\begin{co}}

\newcommand{\ec}{\end{co}}

\newcommand{\brm}{\begin{rem}}

\newcommand{\erm}{\end{rem}}

\newcommand{\f}{\overline{f}}
\newcommand{\F}{\overline{F}}

\usepackage{t1enc}
\def\frak{\mathfrak}

\def\Bbb{\mathbb}
\def\Cal{\mathcal}

\newcommand{\newc}{\newcommand}

\renewcommand{\o}{\circ}

\let\ccdot\cdot
\def\cdot{\hbox to 2.5pt{\hss$\ccdot$\hss}}

\newc{\aR}{\mbox{\boldmath{$ R$}}}
\newc{\aS}{\mbox{\boldmath{$ S$}}}
\newc{\aT}{\mbox{\boldmath{$ T$}}}
\newc{\aW}{\mbox{\boldmath{$ W$}}}
\newcommand{\aX}{\mbox{\boldmath{$ X$}}\hspace{-.2mm}}
\newc{\aD}{\mbox{\boldmath{$ D$}}\hspace{-.2mm}}

\renewcommand{\colon}{\scalebox{1.2}{:}}

\newc{\aK}{\mbox{\boldmath{$ K$}}}
\newc{\aL}{\mbox{\boldmath{$ L$}}}

\newcommand{\ce}{{\Cal E}}

\newcommand{\cO}{{\Cal O}}
\newcommand{\ocO}{{\overline{\Cal O}}}

\newcommand{\cf}{{\Cal F}}

\newcommand{\ct}{{\Cal T}}

\usepackage{amssymb}
\usepackage{amscd}

\newcommand{\nd}{\nabla}

\newcommand{\Rho}{P}
\newcommand{\Up}{\Upsilon}

\newcommand{\End}{\operatorname{End}}

\newcommand{\Ric}{\operatorname{Ric}}
\newcommand{\Sc}{\operatorname{Sc}}

\newcommand{\bs}[1]{\boldsymbol{#1}}
\newcommand{\I}{\bs{I}}

\newcommand{\cT}{{\mathcal T}}

\let\hash=\sharp  

\newcommand{\aM}{\tilde{M}}
\newcommand{\asi}{\tilde{\sigma}}

\newcommand{\cW}{{\Cal W}}

\newcommand{\nn}[1]{(\ref{#1})}

\newcommand{\X}{\mbox{\boldmath{$ X$}}}
\newcommand{\sX}{\mbox{\scriptsize\boldmath{$X$}}}        
\newcommand{\aF}{\boldsymbol{F}}

\newcommand{\h}{\mbox{\boldmath{$ h$}}}
\newcommand{\bg}{\mbox{\boldmath{$ g$}}}

\renewcommand{\S}{\Sigma}

\let\s=\sigma

\newcommand{\J}{{\mbox{\sf J}}}
\newc{\obstrn}[2]{B^{#1}_{#2}}

\newcommand{\rpl}                         
{\mbox{$
\begin{picture}(12.7,8)(-.5,-1)
\put(0,0.2){$+$}
\put(4.2,2.8){\oval(8,8)[r]}
\end{picture}$}}

\newcommand{\lpl}                         
{\mbox{$
\begin{picture}(12.7,8)(-.5,-1)
\put(2,0.2){$+$}
\put(6.2,2.8){\oval(8,8)[l]}
\end{picture}$}}

\usepackage{ifthen}

\newc{\tensor}[1]{#1}
\newc{\Mvariable}[1]{\mbox{#1}}
\newc{\down}[1]{{}_{#1}}
\newc{\up}[1]{{}^{#1}}

\newc{\JulyStrut}{\rule{0mm}{6mm}}
\newc{\midtenPan}{\mbox{\sf S}}
\newc{\midten}{\mbox{\sf T}}
\newc{\midtenEi}{\mbox{\sf U}}
\newc{\ATen}{\mbox{\sf E}}
\newc{\BTen}{\mbox{\sf F}}
\newc{\CTen}{\mbox{\sf G}}

\newcommand{\w}{\mbox{\bf w}} 

\def\sideremark#1{\ifvmode\leavevmode\fi\vadjust{\vbox to0pt{\vss
 \hbox to 0pt{\hskip\hsize\hskip1em
 \vbox{\hsize3cm\tiny\raggedright\pretolerance10000
 \noindent #1\hfill}\hss}\vbox to8pt{\vfil}\vss}}}%
                        
\numberwithin{equation}{section}

\begin{document}
\renewcommand{\today}{}
\title{Boundary calculus\\[.5mm] for\\[.5mm] conformally compact manifolds}
\author{A. Rod Gover \& Andrew Waldron}

\address{Department of Mathematics\\
  The University of Auckland\\
  Private Bag 92019\\
  Auckland 1\\
  New Zealand,  and\\
  Mathematical Sciences Institute, Australian National University, ACT
  0200, Australia} \email{gover@math.auckland.ac.nz}
  
  \address{Department of Mathematics\\
  University of California\\
  Davis, CA95616, USA} \email{wally@math.ucdavis.edu}

\vspace{10pt}

\vspace{10pt}

\renewcommand{\arraystretch}{1}

\begin{abstract} 
  On conformally compact manifolds of arbitrary signature, we use
  conformal geometry to identify a natural (and very general) class of
  canonical boundary problems. It turns out that these encompass and
  extend aspects of already known holographic bulk-boundary problems,
  the conformal scattering description of boundary conformal
  invariants, and corresponding questions surrounding a range of physical
  bulk wave equations.  These problems are then simultaneously solved
  asymptotically to all orders by a single universal calculus of
  operators that yields what may be described as a solution generating
  algebra. The operators involved are canonically determined by the
  bulk ({\it i.e.}\ interior) conformal structure along with a field which
  captures the singular scale of the boundary; in particular the
  calculus is canonical to the structure and involves no coordinate
  choices. The generic solutions are also produced without recourse to
  coordinate or other choices, and in all cases we obtain explicit
  universal formul\ae\ for the solutions that apply in all
  signatures and to a range of fields.  A specialisation of this
  calculus yields holographic formul\ae\ for GJMS operators and
  Branson's Q-curvature.
  
 \vspace{2cm}
  
\end{abstract}

\maketitle
\renewcommand{\arraystretch}{1.5}

\pagestyle{myheadings} \markboth{Gover \& Waldron}{Calculus on conformally compact manifolds}

\maketitle

\tableofcontents

\section{Introduction}\label{intro}

Let $M$ be a $d$-dimensional, compact  manifold with boundary $\Sigma=\partial M$
(all structures assumed smooth).
A  pseudo-Riemannian metric $\g$ on
the interior~$M^+$ of~$M$ is said to be conformally compact if it
extends to $\Sigma$ by~$g=r^2\g$ where~$g$ is non-degenerate up to
the boundary, and $r$ is a  defining function for the
boundary ({\it i.e.}\ $\Sigma$ is the zero locus of $r$, and $d r$ is
non-vanishing along~$\Sigma$).  Assuming the conormal to $\Sigma$ is not null, the restriction of $g$ to $T\Sigma$ in
$TM|_\Sigma$ determines a conformal structure, and this is independent of
the choice of defining function $r$; then $\S$ with this conformal
structure is termed the conformal infinity of $M^+$.  

Conformally compact manifolds provide a framework for relating
conformal geometry, and associated field theories, to the far field
phenomena of the interior (pseudo-)\-Riemannian geometry of one higher
dimension; the latter often termed the {\em bulk}. This idea seems to
have had its origins in the work of Newman and Penrose, for treating
four dimensional spacetime physics, and was further developed by
LeBrun and others~\cite{LeBrunH}.  The seminal work of
Fefferman--Graham~\cite{FGast} developed a variant of this idea,
termed a Poincar\'e(--Einstein) metric, to develop a new approach to
conformal invariant theory. On the other hand key aspects of the
spectral theory for Riemannian conformally compact manifolds were
pioneered by Mazzeo and Mazzeo-Melrose~\cite{Ma-hodge,Ma-unique,MaMe}. Related 
to this is the
treatment of infinite volume hyperbolic manifolds by Patterson, Perry
and others, see~\cite{Perry} and references therein.

More recently there has been intense interest in these structures
motivated first by links to the conjectural AdS/CFT correspondence of
Maldacena and second
by the realisation that they are prototypes for holographic and renormalization ideas~\cite{Mal,AdSCFTreview,Henningson,de Haro,Papadimitriou}. On the side of mathematics, critical progress was made
in~\cite{GrZ} based around a \hypertarget{scalar laplacian}{scalar Laplacian eigen-equation} (given in
our conventions)
\begin{equation}\label{GZeq}
\big(\Delta_{\g} +s(n-s) \big)u=0.
\end{equation}
There earlier spectral results results were further developed, along
with other tools, and used to relate the corresponding scattering
matrix on conformally compact asymptotically Einstein manifolds to
invariant objects on their boundaries at infinity; the latter
including the conformally invariant powers of the Laplacian (the
so-called ``GJMS'' operators of~\cite{GJMS}) and Branson's Q-curvature
of~\cite{tomsharp}.  A sample of the subsequent
and related progress includes~\cite{GuilN,HPT,Juhl,MSV,Vasy}.  In all of
these works the formal asymptotics play a critical role. Indeed
in~\cite{GrZ} it is emphasised at an early point that although the
scattering problem they treat is certainly a global object the main
new results there, and also in~\cite{FGrQ}, concern ``formal Taylor
series statements at the boundary''. The formal calculations are based
around special coordinates adapted to the structure
following~\cite{GrL,GrVol}.

The purpose of the current article is to take full advantage of the
interior conformal structure in order to develop a new universal, and
coordinate independent, approach to a class of natural problems
associated canonically with conformally compact manifolds.  The focus
here is on formal and asymptotic aspects along the boundary, and the
problems themselves are derived directly from the geometric data of
the structure.  The same geometric inputs yield a calculus to treat
these; this applies in all signatures and along any hypersurface that
is a conformal infinity in the sense described above. This includes,
for example, the cases of future and past infinities of de Sitter-like
spaces.  We obtain complete formul\ae\ for the asymptotics to all
orders.  A specialisation of the calculus is applied to give simple
holographic formul\ae\ for the GJMS operators and the Q-curvature.  We
show that the problems considered include those looked at
in~\cite{GrZ}, for example, but simultaneously include their wave
analogues.  Importantly, here these problems are based around a
completely different conceptual schema, which uses conformal tractor
calculus as in~\cite{BEG,CapGoirred,GoIP}; in particular this is used
to show that the problems themselves arise canonically from the
fundamental conformal geometry of the interior/bulk structure.

Our approach applies to weighted tractor fields of arbitrary type.
This provides a universal framework for treating 
 tensor and spinor fields of any tensor type and any conformal weight.
 This is realized by a comprehensive treatment of  Laplace/wave equations for differential forms 
 on Poincar\'e--Einstein structures in~\cite{GLW}.
 In addition, in this Article,  we treat fields that we term log densities (see Section~\ref{dlogd}); 
this connects with problems treated in~\cite{FGrQ,FH}.

A conformally compact manifold may be viewed as a conformal manifold~$(M,c)$ 
with boundary (so $c$ is an equivalence class of conformally
related metrics), that is also equipped with a generalised scale
$\si$.  This scale determines the metric~$\g$ on the interior and also
defines the boundary; $\si$ is a conformal density (see
Section~\ref{dlogd}) with the boundary as its zero locus, and relative
to any metric from the conformal class the function representing $\si$
is a defining function for the boundary.  Thus a conformally compact
structure is precisely the data $(M,c,\si)$. On the other hand there
is a canonical and fundamental conformally invariant, second order,
and vector bundle valued, differential operator $D$ due to
Thomas~\cite{BEG}  that is determined by just the conformal structure
$(M,c)$. This operator applied to $\si$ (and dividing by the
dimension) gives $I:=\frac{1}{d}D \si$, which we term the {\em
  scale tractor}. It will later be evident that there are obvious
representation theoretic interpretations of this as the object
breaking (or rather splitting) a conformal symmetry group pointwise
via the additional structure coming from the scale. The key point here
is that there is an invariant contraction of $I$ with $D$ to yield a
second order operator~$I\cdot D$
that is canonical to the structure. The problems we consider in
Section~\ref{ext} are equations for this degenerate Laplacian operator
(or rather, degenerate wave type operator in non-Riemannian signatures),
with very general and (and probably all) natural boundary conditions.

Apart from the naturality of $I\cdot D$ to pseudo-Riemannian and
conformally compact structures, some comments clarify why it is a
priori both natural and extremely interesting to study in this
context: 
\begin{itemize}
\item First $I\cdot D$ is a degenerate Laplacian, with degeneracy
  precisely along $\Sigma =\partial M$. Away from $\Sigma$, and
  calculated in the metric $\g$, the equation $I\cdot D\, u=0$ is
  $(\Delta_{\g} +s(n-s) )u=0$, as in~\ref{GZeq} with the spectral
  parameter $s$ a fixed dimensional shift of the conformal weight, see
  Section~\ref{degL}, and also~\cite{GoIP}. (Precisely, $I\cdot D\, u=0$
  agrees with $(\Delta_{\g} +s(n-s) )u=0$ in the case of $\g$ having
  constant scalar curvature $\Sc=-d(d-1)$, in general scalar curvature
  enters explicitly.)  However while the operator $\Delta_{\g} +s(n-s) $ is not
  defined along $\Sigma$ where $\g$ is singular, in contrast~$I\cdot
  D$, by dint of its conformal properties, extends to the boundary;
  along $\Sigma$, $I\cdot D$ is a multiple of the first order
  conformal Robin operator of~\cite{cherrier}.
\item On Einstein manifolds $I$ is parallel for the tractor
connection, and powers of $I\cdot D$ give the conformal GJMS operators
 and their Q-curvatures~\cite{powerslap}. This holds in
particular on the interior of Poincar\'e--Einstein manifolds, and this
features in the
asymptotics problems in this setting, enabling a smooth expansion where
otherwise log terms would arise, see Remark \ref{PEcase}.
\item There is strong evidence that the operator $I\cdot D$ is
  fundamental in physics; $I\cdot D \, u=0$ is one of the key
  equations in a system which uniformly controls massive, massless,
  and partially massless free field field
  particles~\cite{GSHI,GSHII,SH}, and extensions include interacting
  terms and supersymmetry~\cite{SHII}. In fact, in the setting of an
  AdS/CFT correspondence, our results essentially amount to explicit
  all orders asymptotic series for ``bulk to boundary propagators''
  for general fields on arbitrary bulk geometries
\end{itemize}

We summarise our results and plan as follows: 
\begin{itemize}
\item In Section~\ref{boundcalc} we prove that $\si$ and $I\cdot D$
  (or the slight variant there\-of,~$\frac{-1}{I^2}I\cdot D$) generate
  a $\frak{g}:=\frak{sl}(2)$ algebra of operators canonically
  determined by the structure. This is  different to  the
  $\frak{sl}(2)$ algebra of GJMS~\cite{GJMS} coming from the Fefferman--Graham ambient
  metric. It leads to different asymptotic expansions which include phenomena 
  not seen from the ambient metric alone.
  The structure  $\frak{g}$ 
  critically underlies the boundary
  calculus. Sections~\ref{tangential ops} and~\ref{ext} apply this calculus, see {\it e.g.},  Lemma ~\ref{gjmsstyle} and 
  Proposition~\ref{gformal}.
\item We show in Section~\ref{tangential ops} that standard identities in the
universal enveloping algebra of $\frak{g}$ lead to the identification of
differential operators on $M$ that act tangentially along $\Sigma$. In
particular we obtain simple new holographic formul\ae\ for the GJMS
operators and Branson's Q-curvature.
\item In Section~\ref{first} and Section~\ref{sec} we consider two
  different extension problems for the equation $I \cdot D \, u=0$, at
  a generic class of weights. The first is Dirichlet in nature while
  the second assumes the solution vanishes along $\Sigma$ to a given
  order that is related to weight. In fact we show that these are the
  only possibilities for smooth ({\it i.e.} power series)
  asymptotics. We obtain complete formal solutions described by a
  Bessel type solution operator; these arise from algebra associated
  to $\frak{g}$. While the presence of Bessel functions is often
  associated with separation of variables approaches ({\it
    e.g.}~\cite{Gubser,Witten}), it is important to note that here the
  asymptotic expansions are given in terms of series in powers of
  $\si$, which is part of the data of the structure itself, and the
  solution operator also has this property.  No coordinates are used,
  the treatment is canonical, and in Proposition~\ref{Kx} we show that
  the solutions are independent of any choices.
\item In Section~\ref{logs} we show the remaining weights can be
  treated by the introduction of expansions involving log terms.
  These cases involve significant subtleties. In particular to match
  the $I\cdot D$ equation with~\nn{GZeq} it is necessary to choose a
  background (second) scale.  After some detailed treatment the
  influence of this choice is made completely clear, and the entire
  solution (family) is controlled by an explicit solution operator.
  The second scale, unlike the generalised scale~$\sigma$, extends to
  the boundary. We show that in fact only  the boundary dependence
  of the second scale influences the solution. This result dovetails
  nicely with the AdS/CFT relationship between boundary
  renormalization group flows and bulk
  geometries~\cite{Gubser,Witten,de Boer:1999xf}.

\item In Section~\ref{total} we show that when the solutions are
  expressed in the internal scale we get results of the usual form as
  in, for example, the work~\cite{GrZ} of Graham-Zworski. In fact it is
  easily verified that our expressions for the asymptotics agree
  with~\cite{GrZ,FGrQ}, at least in the cases of overlap. It should be
  emphasised that throughout our treatment is valid in any signature,
  and whereas the equation $(\Delta_{\g} +s(n-s) )u=0$ has mainly been
  studied with $u$ a function, here we treat the $I\cdot D $ equation
  on sections of (weighted) tractor bundles, and on conformally
  compact structures of any signature.  The operator $I\cdot D $ on tractor bundles
  determines natural equations on tensor and spinor fields, this is
  central to the treatments in~\cite{GSHI,GSHII,SHII}. This means our results can be used to
  deduce results for general tensor and spinor fields
   in the spirit of the
  Eastwood curved translation principle, see~\cite{GoIP} for an early 
  discussion of this. 
  In physics there is currently
  considerable interest in understanding higher spin systems via the
  AdS/CFT machinery~\cite{Klebanov,Giombi}.  
  For  differential forms, in~\cite{AG-BGops}
  Aubry-Guillarmou have shown that subtle global conformal invariants~\cite{BrGodeRham}
  of the conformal infinity are  captured by Poincar\'e-Einstein
  interior problems with suitable boundary behaviour. In the accompanying Article~\cite{GLW},
  we give formal, all order, solutions to those problems and explicit holographic formul\ae\ for 
  the operators yielding those invariants. This  relies on the universal boundary calculus developed here and 
  extends the  curved translation principle.

   \end{itemize}

\thanks{G.\ gratefully acknowledges support from the Royal Society of
  New Zealand via Marsden Grant 10-UOA-113; both authors wish to thank
  their coauthor's institutions for hospitality during the preparation
  of this work. G. wishes to thank Robin Graham for useful discussions
  at the 2011 CIRM meeting ``Analyse G\'eom\'etrique''.}

\section{Conformal geometry and tractor calculus}

Throughout we work on a manifold $M$ of dimension $d\geq 3$. For
simplicity we assume that this is connected and orientable. If this is
equipped with a metric (of some signature $(p,q)$) and $\nabla_a$
denotes the corresponding Levi-Civita connection then the Riemann
curvature tensor $R$ is given by
\[
R(X,Y)Z=\nd_{X}\nd_{Y}Z-\nd_{Y}\nd_{X}Z-\nd_{[X,Y]}Z,
\]
where $X$, $Y$, and $Z$ are arbitrary vector fields. In an abstract
index notation ({\it cf.}~\cite{ot}) $R$ is denoted by $R_{ab}{}^{c}{}_d$,
and $R(X,Y)Z$ is $X^aY^bZ^d R_{ab}{}^{c}{}_d$.
This can be decomposed into the totally trace-free {\em Weyl curvature}
$C_{abcd}$ and the symmetric {\em
Schouten tensor} $\Rho_{ab}$ according to
\begin{equation}\label{Rsplit}
R_{abcd}=C_{abcd}+2g_{c[a}\Rho_{b]d}+2g_{d[b}\Rho_{a]c},
\end{equation}
where $[\cdots]$ indicates  antisymmetrisation over the enclosed
indices. 
Thus $\Rho_{ab}$ is a trace modification of the Ricci tensor 
${\rm Ric}_{ab}=R_{ca}{}^c{}_b$:
$$
\Ric_{ab}=(n-2)\Rho_{ab}+ J g_{ab}, \quad \quad J:=\Rho^a_{~a}.
$$

\subsection{Conformal densities and log densities} \label{dlogd} We
need some results and techniques from conformal geometry.  Recall that
a conformal geometry is a $d$-manifold equipped with an equivalence
class $c$ of metrics (of some fixed signature $(p,q)$) such that if
$g,\widehat{g}\in c$ then $\widehat{g}=e^{2\Up}g$ for some $\Up\in
C^\infty (M)$. Conformally invariant operators on functions or between
(unweighted) tensor bundles are almost non-existent, thus the notion
of covariance and bi-weight is often used. In fact this covariance may
be replaced by true invariance, with conceptual and calculational
simplifications if we introduce conformal densities. For more details
on the approach here see~\cite{CapGoamb,GoPetCMP}.

A conformal structure on $M$ may also be viewed as a smooth ray
subbundle $\cG\subset S^2T^*\!M$ whose fibre over $x\in M$ consists of
conformally related metrics at the point $x$.  The principal bundle
$\pi:\cG\to M$ has structure group $\Bbb R_+$, and so each
representation ${\Bbb R}_+ \ni x \mapsto x^{-w/2}\in {\rm
  End}(\mathbb{C})$ induces a natural (oriented) line bundle on $
(M,[g])$ that we term the conformal density bundle of weight $w\in
\mathbb{C}$, and denote $\ce[w]$. The typical fibre of $\ce[w]$ is
$\mathbb{C}$, which we view as the complexification of an oriented
copy of $\mathbb{R}$, since for $w$ real there is naturally a real ray
subbundle $\ce_+[w]$ of positive elements in $\ce[w]$.
   Note that, by definition, a section
$\tau$ of $\ce[w]$ is equivalent to a function $\underline{\tau}$ on
$\cG$ with the equivariance property
$$
\underline{\tau}(t^2 g,p)= t^{w} \underline{\tau}(g,p), 
$$
where $t\in \mathbb{R}_+$, $g\in c$, and $p\in M$. Note that $g$ is a
section of $\cG$ and the pullback of $\underline{\tau}$ by this is a
function $f$ on $M$ that represents $\tau$ in the trivialisation
determined by $g$. If $\widehat{g}= e^{2\Upsilon} g$, where
$\Upsilon\in C^\infty (M)$ and $\widehat{f}$ is the pullback of $
\underline{\tau}$ via $\widehat{g} $, then
$$
\widehat{f} = e^{w\Upsilon} f.
$$  
Conformal densities are often treated informally as an equivalence
class of such functions.

Each metric $g\in c$ determines a canonical section $\tau\in
\Gamma\ce[1]$, {\it viz.}\ the section with the property that
$\underline{\tau}(g,p)=1$ for all $p\in M$. It follows that there is a
tautological section of $S^2T^*\!M\otimes \ce[2]$ that is termed the
{\em conformal metric}, denoted $\bg$ with the property that any
nowhere zero section $\tau\in \Gamma \ce[1]$ determines a  metric $g\in c$ 
via 
$g:=\tau^{-2}\bg$.
Henceforth the conformal metric~$\bg$ is the
default object that will be used to identify $TM$ with $T^*\!M[2]$ (rather
than a metric from the conformal class) and to form metric traces.

Note that since each $g\in c$ determines a trivialisation of $\ce[w]$,
it also determines a corresponding connection on $\ce[w]$. We write
$\nd$ for this and call it the Levi-Civita connection, since a power
of this agrees with the Levi-Civita connection on $\Lambda^d T^*\!M$
(see {\it e.g.}~\cite{CapGoamb}).

As well as the density bundles $\ce[w]$, the conformal structure also
determines, in an obvious way, bundles $\cf[w]$ induced from the log
representations of $\mathbb{R}_+$. Here also, we may take the {\em
  weight} $w\in \mathbb{C}$, and we refer to $\cf[w]$ as a log density
bundle.  In this case the fibre is $\mathbb{C}$ as the
complexification of $\mathbb{R}$ viewed as an affine space, and the
total space of the bundle is $\cG\times \mathbb{\mathbb{C}}$ modulo 
the equivalence relation
$$
\big((t^2 g,p), y\big)\sim \big( (g,p) , y - w \log t \big) .
$$
In particular $\cf[0]$ is just $\ce$, the trivial bundle. For a general weight $w$,
 a section $\lambda$ of $\cf [w]$ is equivalent to a function $
\underline{\lambda}:\cG\to \mathbb{C} $ with the equivariance property 
\begin{equation}\label{logt}
\underline{\lambda}(t^2 g,p)= \underline{\lambda}(g,p) + w
\log t ~, 
\end{equation}
where, as before, $t\in \mathbb{R}_+$, $g\in c$, and $p\in M$.  As a
smooth structure, $\cf[w]$ is a trivial line bundle. However its
geometric content is not compatible in the usual way with linear
operations.  Nevertheless it is not difficult to see that we may
define such operations on sections, in an obvious way, via their
representative functions on $\cG$. This determines a well defined
notion of adding sections of $\cf[w_1]$ and $\cf[w_2]$, but note that
for $\lambda_1\in \Gamma\cf[w_1]$ and $\lambda_2\in \Gamma\cf [w_2]$,
the sum $\lambda_1+\lambda_2$ is a section of
$\cf[w_1+w_2]$. Similarly we see that, for $w_0\neq 0$, pointwise
multiplication by $w_0$ determines a bundle isomorphism $w_0:\cf[1]\to
\cf[w_0]$.

Note that if $\tau$ is a positive real section of $\ce_+[w]$ and
$\underline{\tau}$ the corresponding equivariant function on $\cG$,
then the composition $\log \o \, \underline{\tau}$ has the
property~\nn{logt}, and so is equivalent to a section of $\cf[w]$ that
we shall denote~$\log \tau$. It is easily seen that a real section of
$\cf[1]$ is $\log \tau$ for some nowhere zero section $\tau\in
\Gamma\ce[1]$. On the space of real sections in $\Gamma\cf[1]$ the
Levi-Civita connection $\nabla$ (as determined by some metric $g\in
c$) determines an operator $\nabla:\Gamma\cf[1]\to \Gamma(T^*M)$ by
$$
\log \tau \mapsto \tau^{-1}\nabla \tau.
$$
This is then extended to $\nabla:\Gamma\cf[w]\to \Gamma(T^*M)$ by
demanding that $\nabla$ commute with complex multiplication, and this
is consistent with also requiring $\nabla $ act linearly over the sum
of log density sections.

Finally in this Section we define the weight operator $\w$. On
sections of a conformal density bundle this is just the linear
operator that returns the weight. So if $\tau\in \ce[w_0]$ then
$$
\w \, \tau = w_0 \tau.
$$
There is a canonical Euler operator $\mbox{\bf v}$ on $\cG$; this is given by
$t\frac{\partial}{\partial t}$ in the coordinates $(t,x)\in
\mathbb{R_+}\times \mathbb{R}^d$ on $\cG$ induced by coordinates $x$
on $M$ and a choice of $g\in c$, where in  these coordinates
$(t^2g,p)\mapsto (t, x(p))$. The weight operator arises from the
restriction of this to suitably equivariant sections, for example
$\underline{ \w \tau} = \mbox{\bf v} \underline{\tau}$. Thus $\w$ satisfies a
Leibniz rule: for example if $\tau_1\in \Gamma\ce[w_1]$ and $\tau_2\in
\Gamma\ce[w_2]$ then $\w \, (\tau_1\tau_2) = (w_1+w_2)\tau_1\tau_2 $. 
This also determines the action of $\w$ on log densities: if $
\lambda\in \Gamma\cf[w_0] $ then
\begin{equation}\label{law of the logs}
\w \, \lambda = w_0.
\end{equation}
In view of the Leibniz property we may also write $[\w , \lambda ] = w_0$ where
$[ ~,~]$ indicates an operator commutator bracket.

\subsection{Elements of tractor calculus}\label{trc}
The systematic construction of conformally invariant differential
operators can be facilitated by {\em tractor calculus}~\cite{BEG}. This is based
around a bundle and connection that is linked to, and equivalent to, the
normal conformal Cartan connection of Elie Cartan~\cite{CapGoirred,CapGoTAMS}.

On a conformal $d$-manifold $(M,c)$, the (standard) tractor bundle~$\ct$ 
(or~$\ce^A$ as the abstract index notation) is a canonical rank
$d+2$ vector bundle equipped with the canonical (normal) tractor
connection $\nabla^\ct$.  A choice of metric $g\in c$ determines an
isomorphism
\begin{equation}\label{split}
\mathcal{T} \stackrel{g}{\cong} \ce[1]\oplus T^*\!M[1]\oplus
\ce[-1] ~.
\end{equation}
In the following we shall frequently use~\nn{split}. Sometimes this
will be without any explicit comment but also we may write for example
$T\stackrel{g}{=}(\si,~\mu_a,~\rho)$, or alternatively
$[T]_g=(\si,~\mu_a,~\rho)$, to mean $T$ is an invariant section of
$\ct$ and $(\si,~\mu_a,~\rho) $ is its image under the isomorphism
\nn{split}. In terms of this splitting the tractor connection is given by 
\begin{equation}\label{trconn}
\nd^{\ct}_a
\left( \begin{array}{c}
\si\\\mu_b\\ \rho
\end{array} \right) : =
\left( \begin{array}{c}
    \nabla_a \si-\mu_a \\
    \nabla_a \mu_b+ \bg_{ab} \rho +\Rho_{ab}\si \\
    \nabla_a \rho - \Rho_{ac}\mu^c  \end{array} \right) .
\end{equation}
Changing to a conformally related metric
$\widehat{g}=e^{2\Up}g$  gives a different
isomorphism, which is related to the previous by the transformation
formula
\begin{equation}\label{transf}
\widehat{(\si,\mu_b,\rho)}=(\si,\mu_b+\si\Up_b,\rho-\bg^{cd}\Up_c\mu_d-
\tfrac{1}{2}\si\bg^{cd}\Up_c\Up_d),  
\end{equation}
where $\Upsilon_a$ is the one-form $d\Up$. It is straightforward to verify that the
right-hand-side of~\nn{trconn} also transforms in this way and this
verifies the conformal invariance of~$\nabla^\ct$. In the following we
will usually write simply $\nabla$ for the tractor connection. Since
it is the only connection we shall use on $\ct$, its dual, and tensor
powers, this should not cause any confusion.

There is also a conformally invariant {\em tractor metric} $h$ on
$\mathcal{T}$. This is given (as a quadratic form on $T^A$ as above) by
\begin{equation}\label{trmet}
(\si,\, \mu,\, \rho)\mapsto \bg^{-1}(\mu,\mu)+2\si \rho = : h(T,T)=h_{AB}T^AT^B~; 
\end{equation}
it is preserved by the connection. 
We shall often write $T^2$ as a shorthand for the right hand side of this display.
Note that this has signature
$(p+1,q+1)$ on a conformal manifold $(M,c)$ of signature $(p,q)$. The tractor metric
$h_{AB}$ and its inverse $h^{AB}$ are used to identify $\ct$ with its
dual in the obvious way.

Tensor powers of the standard tractor bundle $\ct$, and tensor parts
thereof, are vector bundles that are also termed tractor bundles. We
shall denote an arbitrary tractor bundle by $\ce^\Phi$ and write
$\ce^\Phi[w]$ to mean $\ce^\Phi\otimes \ce[w]$; $w$ is then said to be
the weight of $\ce^\Phi[w]$. In the obvious way, the operator $\w$ of Section~\ref{dlogd} 
is extended to sections of weighted tractor bundles $\ce^\Phi[w]\ni f$ by
$$
\w f = w f\, .
$$

Whereas the tractor connection maps sections of a weight 0 tractor
bundle~$\ce^\Phi$ to sections of $T^*\!M \otimes \ce^\Phi$, there is a
conformally invariant operator which maps between sections of weighted
tractor bundles. This is the Thomas-D (or tractor-D) operator
$$
D_A :  \Gamma \ce^\Phi [w]\mapsto \Gamma (\ce^A\otimes \ce^\Phi[w-1]),
$$
given in a scale $g$ by
\begin{equation}\label{Dform}
[D_A V ]_g =\left(\begin{array}{c} (d+2w-2)\w V\\
(d+2\w-2) \nabla_a V\\
-(\Delta+ \J\w) V   \end{array} \right) ,
\end{equation}
where $\Delta=\bg^{ab}\nabla_a\nabla_b$, and $\nabla$ is the coupled
Levi-Civita-tractor connection~\cite{BEG,T}. 

A key point to emphasise here is that the Thomas-D operator is a
fundamental object in conformal geometry. On a conformal manifold the
tractor bundle is ``as natural'' as the tangent bundle. On the other
hand the tractor-D operator on densities on $\Gamma\ce[1]$ basically
defines the tractor bundle, see~\cite{CapGoirred}.

\subsection{Almost Einstein structures and Poincar\'e--Einstein
  manifolds}
We recall here some facts from the literature that we will need later. 
These partly motivate our overall approach.  The first is that conformal
geometry has a strong bias toward Einstein metrics
\cite{Sasaki}. Following~\cite{GoNur} we state this as follows.
\begin{theorem}\label{cein}
  On a conformal manifold $(M^d,c)$ (of any signature) there is a 1-1
  correspondence between conformal scales $\si\in \Gamma\ce[1]$, such
  that $g^\si=\si^{-2}\bg$ is Einstein, and parallel standard tractors
  $I$ with the property that $X_A I^A$ is nowhere vanishing. The
  mapping from Einstein scales to parallel tractors is given by
  $\si\mapsto \frac{1}{d}D_A \si$ while the inverse is $I^A \mapsto
  X^AI_A$.
\end{theorem}
 The statement as here is easily verified
using~\nn{trconn}, or may be viewed as an easy consequence of the {\em
  definition} of the tractor connection from~\cite{BEG}. However the
normal tractor connection is canonical from other points of view, and
so this suggests a deep link between conformal geometry and Einstein
structures, or more generally to almost Einstein structures
\cite{GoIP} as follows.
\begin{definition}\label{aedef}
  We say that a conformal manifold $(M^d,c)$, $d\geq 3$, is {\em
    almost Einstein} if it is equipped with a non-zero parallel standard
  tractor $I$.
\end{definition}
Since this condition is equivalent to a certain holonomy reduction of
the normal tractor connection, it follows at once that in general a
conformal manifold $(M,c)$ will not admit an almost Einstein
structure. It is only a slight generalisation of the Einstein
condition, and in fact it follows easily from the Theorem above, that
on an almost Einstein manifold $\si$ is non-zero on an open dense
set~\cite{GoNur}.  Note that the zero locus $\mathcal{Z}(\si)$ of the
``scale'' $\si$ is a conformal infinity, if non-empty. 

On a Riemannian conformally compact manifold, if the defining function
(as in Section ~\ref{intro}) satisfies $|d r|^2_g=1$ along $M$, the
sectional curvatures of $\g$ tend to $-1$ at infinity and the structure is
said to be {\em asymptotically hyperbolic (AH)}~\cite{Ma-hodge}.  The
model geometry here is the Poincar\'e hyperbolic ball and thus the corresponding
metrics are sometimes called Poincar\'e metrics.  Generalising the
hyperbolic ball in a more strict way, one may suppose that the interior
conformally compact metric $\g$ is Einstein with the normalisation
$\Ric (\g)=-n \g$, where $n=d-1$, and in this case the structure is
said to be Poincar\'e-Einstein (PE); in fact PE manifolds are
necessarily asymptotically hyperbolic.
As noted in~\cite{GoPrague}, and then discussed in more detail in
\cite{GoIP}, we have the following result.
\begin{proposition}
  A Poincar\'e--Einstein manifold is a (Riemannian signature) almost
  Einstein manifold $(M,c,I)$ with boundary satisfying $\partial
  M=\mathcal{Z}(\si)$, and such that $I^2=1$. 
\end{proposition}

It follows
that the notion of an almost Einstein structure provides an approach
for treating Poincar\'e--Einstein structures that is closely related
to the Cartan/tractor calculus~\cite{Goal}.

Surprisingly, for many problems the same techniques apply to {\em any
  conformally compact structure}, and this is a critical observation
we make and use here.

\subsection{The scale tractor} \label{scale} To avoid awkward language
we shall use the term {\em  (generalised) scale} for any section $\si\in
\Gamma\ce[1]$ that is not identically zero; only in the case that
$\si$ is nowhere zero is this a {\em true scale}, so that $\si^{-2}\bg$
is a metric.
\begin{definition}\label{aedef1} For $\si$ a scale, on any conformal manifold, 
we call
$$
I_A:=\frac{1}{d}D_A \si
$$
the corresponding {\em scale tractor}. Note that $\si=X^AI_A$.
\end{definition}

So an almost Einstein manifold has a parallel scale tractor.
Specialising further: In the model of the hemisphere equipped with
its standard hyperbolic metric,  $I$
may be identified with the vector in $\mathbb{R}^{d+2}$ the fixing of
which reduces the conformal group $SO(d+1,1)$ to a copy of $SO(d,1)$; the latter
with orbit spaces the bounding equatorial conformal infinity, and the interior~\cite{Goal}.

One congenial feature of the scale tractor $I$ is that its squared
length gives a generalisation of the scalar curvature as follows (see
\cite{Goal} which deals with Riemannian signature but the result
applies in any case).  
First note that using~\nn{trmet} and~\nn{Dform} we have 
$$I^2 \stackrel{\g}{=}
-\frac{2\si}{d} (\Delta \si+ \J \si) + (\nabla^a\si)(\nabla_a \si).$$
Thus if $I^2$ is nowhere vanishing then $\si$ is non-vanishing on an
open dense set in $M$.
 Where $\si$ is non-zero let us write $\g:=
\si^{-2}\bg$. Then, at any point~$q$ with $\si(q)\neq 0$, 
\[
I^2 \stackrel{\g}{=} -\frac{2}{d}\si^2 \J=
-\frac{2}{d} J^{\g},
\] 
since $\si$ is preserved by the metric it determines. Note that here
$\J=\bg_{ab}P^{ab}$, whereas $J^{\g}=\g_{ab}P^{ab}$ is the trace of
the Schouten tensor {\em as calculated in the scale} $\g$.

\hypertarget{locus}{Furthermore} note that if $I^2$ is nowhere zero then at any $p\in M$
such that $\si(p)=0$, we have $\nabla_a \si\neq 0$; this follows
since $\bg^{ab}(\nabla_a \si)(\nabla_b \si) $ is of the form $I^2-2\si
\rho$ (for some smooth density $\rho$).  It follows that the zero
locus $\mathcal{Z}(\si)$ of~$\si$ is either empty or forms a smooth
hypersurface ({\it i.e.} a smoothly embedded codimension 1
submanifold), {\it cf.}~\cite{Goal}.

 Let us summarise
aspects of this discussion.
\begin{lemma}\label{sc} Let $I$ be a scale tractor on any conformal manifold.
Where the scale $\si:=X^AI_A$ is non-zero we have 
$$
I^2:=I^AI_A=-\frac{\Sc^{\g}}{d(d-1)},
$$
with $\g:=\si^{-2}\bg$. If $I^2$ is nowhere
zero then $\si$ is non-zero on an open dense set, and
$\mathcal{Z}(\si)$ (if non-empty) is a smooth hypersurface.
 \end{lemma}

 In this language asymptotically hyperbolic manifolds are conformally
 compact manifolds $(M,c,\si)$ with the scale tractor satisfying
 $I^2=1$ along $\partial M$.  On the other hand on any manifold with
 $(M,c,\si)$ (where $\si$ is a generalised scale), $I^2=constant$ means
 that $\g$ has constant scalar curvature where defined, with, in
 particular, $I^2=0$ meaning that $\g$ has zero scalar curvature.  On
 a given conformal manifold it is evident that seeking scales such
 that~$I^2$ is constant is a natural problem generalising the Yamabe
 and Loewner-Nirenberg problems. Following~\cite{Goal} we say a
 conformal manifold with constant $I^2$ is almost scalar constant
 (ASC). In the case of a Riemannian signature ASC manifold with $I^2
 \neq 0$, $\mathcal{Z}(\si)$ can only be non-empty if $I^2$ is
 positive, and if so the structure is asymptotically hyperbolic along
 $\mathcal{Z}(\si)$.  On the other hand, in other signatures there are
 more possibilities. For example de Sitter space is a Lorentzian
 signature ASC structure with $I^2$ negative. The zero locus
 $\mathcal{Z}(\si)$ is the boundary and comprises two components which
 are interpreted as a past conformal infinity and a future conformal
 infinity; each boundary component is spacelike (as follows easily
 from the sign of
 $I^2$).

 \subsection{The canonical degenerate Laplacian} \label{degL} Putting
 the above results together we note that there is a completely
 canonical differential operator on conformal manifolds equipped with
 a scale. Namely 
$$
I\cdot D:= I^AD_A.
$$
This acts on any weighted tractor bundle, preserving its tensor type
but lowering the weight:
$$
I\cdot D: \ce^\Phi[w]\to  \ce^\Phi[w-1].
$$
Now there is a rather subtle point, the importance of which will soon
be clear, namely that, since the tractor-D operator $D_A$ is
conformally invariant, the operator $I\cdot D$ is conformally invariant
except  by the explicit coupling to the scale tractor.

Expanding $I\cdot D$ in terms of some background metric $g$, we have 
$$
I\cdot D \stackrel{g}{=}
\left(\begin{array}{rrr}
-\frac{1}{d}(\Delta \si +\J\si) & \nabla^a \si & \si
\end{array}\right)
\left(\begin{array}{lll}
\w (d+2\w-2) \\  \nabla_a ( d+2\w-2) \\ -(\Delta +\J \w)
\end{array}\right),
$$
and as an operator on any density or tractor bundle of some weight, {\it e.g.}\ $w$ so
$\ce^\Phi[w]$, each occurrence of $\w$ evaluates to $w$. 
So then 
\begin{equation}\label{degLa}
I\cdot D \stackrel{g}{=} -\si \Delta
+(d+2w-2)[(\nabla^a \si)\nabla_a - \frac{w}{d} (\Delta \si)] -\frac{2w}{d}(d+w-1)\si \J~.
\end{equation}
Now if we calculate in the metric $\g=\si^{-2}\bg$, away from  the
zero locus of $\si$, and trivialise the densities accordingly,  then $\si=1$ in such a scale and we have

$$
I\cdot D \stackrel{\g}{=}-  \Big(\Delta^{\g} + \frac{2w(d+w-1)}{d}\J^{\g} \Big).
$$
In particular if $\g$ satisfies $\J^{\g}=-\frac{d}{2}$ ({\it i.e.}\
$\Sc^{\g}=-d(d-1)$ or equivalently $I^2=1$) then, relabeling $d+w-1=:s$ and $d-1=:n$, we have

\begin{equation}\label{IdotDsc}
I\cdot D \stackrel{\g}{=} -  \big(\Delta^{\g} + s(n-s) \big),
\end{equation}
as commented in the \hyperlink{scalar laplacian}{Introduction}. 

On the other hand, looking again to~\nn{degLa}, we see that along
$\Sigma = \mathcal{Z}(\si)$ (assumed non-empty) the operator $I\cdot
D$ is first order. For example (see~\cite{Goal}), if the manifold is
ASC with $I^2=1$, then along $\Sigma$ 
$$
I\cdot D = (d+2w-2) \delta_n~,
$$
on $\ce^{\Phi}[w]$, where $\delta_n$  is the conformal Robin operator,
$$
\delta_n \stackrel{g}{=}  n^a\nabla^g_a - w H^g ,
$$
of~\cite{cherrier,BrGoOps} (twisted with the tractor connection);
here $n^a$ is a unit normal and~$H^g$ the mean curvature, as measured
in the metric $g$.

To further connect with other canonical problems studied in the
literature, it is interesting to explicate the operator $I\cdot D$ on
log densities.  In particular let $U\in \Gamma \cf[w]$, then
$$
I\cdot D \, U\stackrel{g}{=}
 \big[-\si \Delta
+(d-2)(\nabla^a \si)\nabla_a 
\big]\ U
 -\frac{w}{d}\big[(d-2)\Delta \si+2(d-1)\si \J\big]\,  .
$$
Thus in the preferred scale
\begin{equation}\label{FGrQform}
I\cdot D \, U\stackrel{\g}{=}-\, \Delta^{\g} \, U +w(d-1)\, .
\end{equation}
This last operator at $w=-1$ is precisely the one studied in~\cite{FGrQ}.

\section{A boundary calculus for the degenerate Laplacian}
\label{boundcalc}

Let $(M,c)$ be a conformal structure of dimension $d\geq 3$ and of any
signature. 
Given $\si$ a section of $\ce[1]$, write $I_A $ for the
corresponding scale tractor. That is $I_A= \frac{1}{d}D_A\si $.
Then $\si=X^AI_A$.

\subsection{The ${\frak sl}(2)$}
 Suppose that $f\in \ce^\Phi[w]$, where
$\ce^\Phi$ denotes any tractor bundle. Select $g\in c$ for the purpose
of calculation, and write $I_A\stackrel{g}{=}(\si,~\nu_a,~\rho)$ to
simplify the notation.  Then using $\nu_a=\nabla_a\si$,  we have
\begin{equation*}
\begin{split}
I\cdot D\big(\si f\big) =  (d+2w)\big((w+1) \si \rho f +\si \nu_a\nabla^a f+ 
f \nu_a\nu^a\big)\quad\\[2mm] \qquad\qquad  -\si\big(\si \Delta f+2\nu_a\nabla^a f +f\Delta \si +(w+1)\J\si f\big)\, ,
\end{split}
\end{equation*}
while
$$ - \si\,  I\cdot D\,  f = -\si(d+2w-2)  \big(w\rho f +
\nu_a\nabla^a f\big) +\si^2 (\Delta f + w \J f)\, .
$$
So, by virtue of the fact that $\rho=-\frac1d(\Delta\si+\J\si)$,  we have
$$
[I\cdot D,\si] f= (d+2w) (2\si\rho  +\nu_a\nu^a) f .
$$
Now $I^AI_A=I^2\stackrel{g}{=}2\si\rho+\nu_a\nu^a$, whence
the last display simplifies to
$$
[I\cdot D,\si] f =(d+2w) I^2 f . 
$$ 
Denoting by $\w$ the weight operator on
tractors, we have the following. 
\begin{lemma}\label{slem}
Acting on any section of a weighted tractor bundle we have 
$$
[I\cdot D,\si]  =  I^2 (d+2\w)  ,
$$
where $\w$ is the weight operator.
 \end{lemma}

\begin{remark}\label{remarb}
	 	A  similar  computation  to  above  shows  that, more generally,
	 	$$
	 	I\cdot  D  \big(\sigma^\alpha  f\big)  -  \sigma^\alpha   
		I\cdot  D  f  =  -\sigma^{\alpha-1}\alpha\,   I^2 (d+2\w+\alpha-1)   f\,  ,
	 	$$
	 	for  any  constant  $\alpha$.
\end{remark}

\begin{remark}
The identity of Lemma~\ref{slem}, and the algebra it generates (as below), is
unchanged upon adding to the operator $I\cdot D$ an additional term of
the form:
$$I\cdot D\to I\cdot D + \cW^{\overbrace{\hash\cdots\hash}^{k{\rm \  times}}}\, ,$$
where 
$\cW$ is any weight~$-1$ tractor tensor in $\Gamma(\otimes^k\End(\cT)
)$ and $\hash$ denotes the natural tensorial action of endomorphisms
on tractor sections (to the right), so for example on a rank one
tractor $T^A$
$$
\cW^\hash T^A:=-\cW^{A}{}_B T^B\, .
$$
Hence the canonical operator $I^A \slashed D_A$ of Theorem 4.7
of~\cite{Goal}, for example, obeys the same algebra as $I\cdot D$ with
$\sigma$ and the weight operator~$\w$. Thus the extension problem
uncovered there for the $W$-tractor for almost Einstein structures is
also solved formally by the algebraic methods introduced in
Section~\ref{ext} below.
\end{remark}

 The operator $I\cdot D$ lowers conformal weight by 1. On the other
 hand, as an operator (by tensor product) $\si$ raises conformal
 weight by 1.  We can record this by the commutator relations
$$
[\w,I\cdot D]=-I\cdot D\quad \mbox{and} \quad [\w, \si] =\si, 
$$
so with the Lemma we see that the operators $\si$, $I\cdot D$,
 and $\w$, acting on weighted scalar or tractor fields, generate an
${\frak sl}(2)$ Lie algebra, provided $I^2$ is nowhere vanishing. It
is convenient to fix a normalisation of the generators; we record
this and our observations as follows. 
\begin{proposition}\label{slprop} Suppose that $(M,c,\si)$ 
is such that
  $I^2$ is nowhere vanishing.   Setting $x:=
  \si$, $y:= -\frac{1}{I^2}I\cdot D$, and $h:=d+2\w$ we obtain the
  commutation relations
$$
[h,x]=2x, \quad [h,y]=-2y, \quad [x,y]=h, 
$$
of standard  ${\frak sl}(2)$ generators. 

In the case of $I^2=0$ the result is an
 In\"{o}n\"{u}-Wigner contraction of the~${\frak sl}(2)$: 
$$
[h,x]=2x, \quad [h,y]=-2y, \quad [x,y]=0, 
$$ 
where  $h$ and $x$ are as before, but now $y=-I\cdot D$. 
\end{proposition}

\noindent
Subsequently $\frak{g}$ will be used to denote this (${\frak sl}(2)$)
Lie algebra of operators.
From Proposition~\ref{slprop}  (and  in  concordance  with  remark~\ref{remarb}) follow some useful identities in the universal
enveloping algebra $\cU(\frak{g})$.
\begin{corollary}\label{corpid}
\begin{eqnarray}
&[x^k,y]=x^{k-1} k(d+2\w+k-1) = x^{k-1} k (h+k-1)&\nonumber \\[1mm]{} & {\rm  and }&\label{pid}\\[1mm]{} & \ [x,y^k]=y^{k-1}
k(d+2\w-k+1)=y^{k-1}k(h-k+1)\, .&\nonumber
\end{eqnarray}
\end{corollary}

\section{Tangential operators and holographic formul\ae\ }
\label{tangential ops}

Suppose that $\si\in \Gamma\ce[1]$ is such that $I_A=\frac{1}{d}D_A
\si$ satisfies that $I^AI_A= I^2$ is nowhere zero.  
As explained in \hyperlink{locus}{Section~\ref{scale}},
 the zero locus $\mathcal{Z}(\si)$ of~$\si$ is then either empty or
forms a smooth hypersurface ({\it i.e.} a smoothly embedded
codimension~1 submanifold), {\it cf.}~\cite{Goal} where the Riemannian
signature case is discussed.

Conversely if $\Sigma$ is any smooth oriented hypersurface then, at
least in a neighbourhood or $\Sigma$, there is a smooth defining
function $f$.  That is,~$\Sigma$ is the zero locus of a smooth function
$s$, with $ds$ nowhere vanishing along~$\Sigma$.  Now take
$\si\in \Gamma \ce[1]$ to be the unique density which gives $s$ in the
trivialisation of $\ce[1]$ determined by some $g\in c$. It follows then
that $\Sigma =\mathcal{Z}(\si)$ and $\nabla^g \si $ is non-zero at all
points of $\Sigma$. If $\nabla^g \si $ is nowhere null along $\Sigma$
we say $\Sigma$ {\em is nowhere null} 
and then $I^2$ is nowhere vanishing in a
(possibly smaller) neighbourhood of $\Sigma$, and we are in the
situation of the previous paragraph. We shall call such a $\si$ a {\em
  defining density} for $\Sigma$, and to simplify the discussion we
shall take $M$ to agree with this neighbourhood of $\Sigma$ and write $M_o:= M\setminus \Sigma$. Until
further notice~$\si$ will mean such a section of $\ce[1]$ with
$\Sigma=\mathcal{Z}(\si)$ non-empty and nowhere null.  Note that $\Sigma$ has a
conformal structure $c_\Sigma$ induced in the obvious way from $(M,c)$
and is a conformal infinity for the metric $\g:= \si^{-2}\bg$ on
$M\setminus \Sigma$. 

\subsection{Tangential operators}\label{tanS}
Suppose that $\si$ is a defining density for a hypersurface $\Sigma$.
Let $P: \Gamma\ce^\Phi[w_1]\to \Gamma\ce^\Phi[w_2]$ be some linear
 operator in a neighbourhood of $\Sigma$. 
 We shall say that $P$ {\em acts
  tangentially (along $\Sigma$)} if $P \o \si =\si \o\widetilde P $, where $\widetilde P
:\Gamma\ce^\Phi[w_1-1]\to \Gamma\ce^\Phi[w_2-1]$ is some other linear
 operator on the same neighbourhood.
The point is that for a tangential operator $P$ we have 
$$
P (f+\si h)= P f+ \si \widetilde P h.
$$  
Thus along $\Sigma$ the operator $P$ is insensitive to how $f$ is
extended off $\Sigma$.  It is easily seen that, for example, in the
case that $P$ is a tangential differential operator, there is a formula for $P$,
along $\Sigma$, involving only derivatives tangential to $\Sigma$, and
the converse also holds.

Using Corollary~\ref{corpid} we see at once that certain powers of
$I\cdot D$ act tangentially on appropriately weighted tractor bundles.
We state this precisely. Suppose that $\Sigma$ is a (nowhere null) hypersurface in a
conformal manifold $(M^{n+1},c)$, and $\si$ a defining density for
$\Sigma$. Then recall $\Sigma =\mathcal{Z}(\si)$ and $I^2$ is nowhere
zero in a neighbourhood of $\Sigma$, where $I_A:=\frac{1}{n+1}D_A\si$
is the scale tractor.  The following holds.
\begin{theorem}\label{tanth}
  Let $\ce^\Phi$ be any tractor bundle and $k\in \mathbb{Z}_{\geq 1}$.
  Then, for each $k\in \mathbb{Z}_{\geq 1}$, along $\Sigma$
\begin{equation}\label{tan}
P_k: \Gamma\ce^\Phi[\frac{k-n}{2}]\to \Gamma\ce^\Phi[\frac{-k-n}{2}]
\quad \mbox{given by}\quad 
P_k:= \Big(-\frac{1}{I^2} I\cdot D\Big)^k
\end{equation}
is a tangential differential operator, and so determines a canonical
differential operator $P_k: \Gamma\ce^\Phi[\frac{k-n}{2}]|_\Sigma \to
\Gamma\ce^\Phi[\frac{-k-n}{2}]|_\Sigma$.
\end{theorem}
\begin{proof}
  The $P_k$ are differential by construction. Thus the result is
  immediate from Corollary~\ref{corpid} with $\widetilde P =
  \Big(-\frac{1}{I^2} I\cdot D\Big)^k$.
\end{proof}
\begin{remark}
  As in the Theorem above, we shall use the same notation $P_k$ to
  mean the tangential differential operator defined by~\nn{tan}, and
  also the operator $P_k: \Gamma\ce^\Phi[\frac{k-n}{2}]|_\Sigma \to
 \Gamma \ce^\Phi[\frac{-k-n}{2}]|_\Sigma$ that this determines. The meaning
  will be clear by context.
\end{remark}

Recall we write $n=d-1$. The operator $P_n$ is said to be of {\em
  critical order}, since its order along $\Sigma$ 
is $n$, the dimension of
$\Sigma$. As an operator on functions we have the following.
\begin{proposition}\label{kills1}
The differential operator on functions
$$
P_n:\Gamma\ce_\Sigma \to \Gamma\ce_\Sigma[-n]
$$
annihilates constant functions.
\end{proposition}
\begin{proof}
  Consider $ P_n 1 $. Since the ambient operator $P_n$ is tangential there is no
  loss of generality in extending the constant function $1$ on
  $\Sigma$ to the constant function $1$ on $M$ and calculating $[P_n 1]_\Sigma$. 
But $D_A 1=0$, and so $P_n 1=0$.
\end{proof}

\subsection{Holographic formul\ae\ for the GJMS operators and
  Q-curvature} \label{holog}
Theorem~\ref{tanth} does not establish whether or not the $P_k$ are
non-trivial. In cases where they are non-trivial then they will be
intrinsic to the hypersurface if the metric $\g$ is determined by the
conformal structure on $\Sigma$ to sufficient order.  

In the case where $\g$ is  Einstein we can, for the most
part, fill in these details.  Note that $\g$  Einstein
implies $I^2$ is a non-zero constant. By a constant scaling of $\si$
we may take this to be $\pm 1$ (and recall $\Sc^{\g}= -n(n+1)$ $I^2$), so without
loss of generality we assume this.

First we record the following.
\begin{proposition}\label{zero}
  Suppose that $(M_o,\g)$ is Einstein, with $\Sc^{\g}= \mp
  n(n+1)$. Then for $k$ odd $P_k:
  \Gamma\ce^\Phi[\frac{k-n}{2}]|_\Sigma \to
  \Gamma\ce^\Phi[\frac{-k-n}{2}]|_\Sigma$ is the zero operator. For~$k$ 
  even the operator has leading term ($\ (-1)^k \big((k-1)!!\big)^2$ times) $\Delta_\Sigma^{\frac k2}$, (up to
  the sign~$\pm$) where $\Delta_\Sigma$ is the intrinsic Laplacian for
  $(\Sigma,g_\Sigma) $ with $g_\Sigma\in c_\Sigma$.
\end{proposition}

The picture can be refined significantly for operators on
densities.  We write $\ce_\Sigma[w]$ for the intrinsic weight $w$
conformal density bundle of $(\Sigma,c_\Sigma)$, and note that this
may be canonically identified with $\ce[w]|_\Sigma$.
\begin{theorem}\label{tanein}
  Suppose that $(M_o,\g)$ is Einstein, with $\Sc^{\g}=
  \mp n(n+1)$.  Then the operator on densities $P_k: \Gamma\ce_\Sigma
  [\frac{k-n}{2}]\to \Gamma\ce_\Sigma [\frac{-k-n}{2}]$ satisfies the following:
\begin{itemize}
\item 
For $k$ odd $P_k$ is the zero operator.

\item For  even $k$, with $k\leq n$,  $ P_k $ is  (~$(-1)^k \big((k-1)!!\big)^2$ times) the order $k$ GJMS operator $\mathcal{P}_k$ on
  $(\Sigma,c_\Sigma)$. 
\end{itemize}
\end{theorem}

The proof of this is given in Section~\ref{aefg} below.

\begin{remark}\label{holo}
  For the ranges as discussed, (and in the spirit of {\it
    e.g.}~\cite{GrJ,Manvelyan,Diaz}) we can view the $k$-even $P_k$ as
  giving {\em holographic formul\ae\ } for the GJMS operators. These
  are given as an explicit formula in terms of the ambient almost
  Einstein (or Poincar\'e--Einstein) space which has one higher
  dimension than the space where the operator acts.

  In the case of $n$ odd the Theorem does not recover all the
  operators of~\cite{GJMS}. This is just an issue of the extent to
  which  the boundary conformal structure determines the ambient
  metric $\g$. For example if $\g$ is the formal Poincar\'e metric in
  the sense of~\cite{FGrnew,GrZ} then $P_k$ will be intrinsic for all
  $k$, and it is not difficult to argue (following~\cite{GrZ}) that
  for all $k\in 2\mathbb{Z}_{\geq 1}$ these are the GJMS operators.
\end{remark}

We can now give a similar holographic formula for the $Q$-curvature.
\begin{theorem}\label{Qthm}
  Suppose that $(M_o^{n+1},\g)$ is Einstein, with $\Sc^{\g}= \mp n(n+1)$
  and $n$ even.  
Let $\mu\in \Gamma\ce_\Sigma[1]$ be a scale for
  $(\Sigma,c_\Sigma)$. This determines a metric $g^\mu_\Sigma\in
  c_\Sigma$ and the corresponding (Branson) Q-curvature of
  $(\Sigma,g^\mu_\Sigma )$ is given by
\begin{equation}\label{Q}
Q^{g^\mu_\Sigma} =  \frac{1}{\big((n-1)!!\big)^2}\Big[ \Big(-\frac{1}{I^2} I\cdot D\Big)^n \log \tilde{\mu} \Big]\Big|_\Sigma\, ,
\end{equation}
where $\tilde{\mu}$ is any smooth extension of $\mu$ to a section of
$\ce[1]$ in a neighbourhood of $\Sigma$.
\end{theorem}

\begin{proof}
  Suppose that $\widehat{\tilde{\mu}}$ is a different smooth extension of
  $\mu$.  Then $\widehat{\tilde{\mu}} = f \tilde{\mu} $, for some smooth
  function $f$ with the property that $f|_\Sigma =1$. Thus
$$
\Big[ \Big(-\frac{1}{I^2} I\cdot D\Big)^n \log
\widehat{\tilde{\mu}}\Big]\Big|_\Sigma = \Big[ \Big(-\frac{1}{I^2} I\cdot D\Big)^n \log
\tilde{\mu} \Big]\Big|_\Sigma + \Big[ \Big(-\frac{1}{I^2} I\cdot D\Big)^n \log
f \Big]\Big|_\Sigma. 
$$
But
$$
\Big[ \Big(-\frac{1}{I^2} I\cdot D\Big)^n \log
f \Big]\Big|_\Sigma = \big[P_n f\big]\big|_\Sigma = \big[P_n 1\big]\big|_\Sigma =0,
$$
since the critical operator $P_n$ is tangential and, by
Proposition~\ref{kills1}, annihilates constants. Thus we see that the
right-hand-side of~\nn{Q} is independent of how $\mu$ is extended to
$\tilde{\mu}$ (that is, $y^n\o \log$ is a tangential operator on
$\ce[1]$).

Now the proof follows easily by a Fefferman-Graham ambient metric
argument analogous to the proof of Theorem~\ref{tanein}, using also
the result of Fefferman-Hirachi~\cite{FH} that, in terms of the
Fefferman--Graham ambient metric, the Q-curvature is given by
$\al^{n/2}_{\tilde{\Sigma}} \log \boldsymbol{\mu}|_{\tilde{\Sigma}}$
where $\boldsymbol{\mu}$ is a homogeneous degree 1 function on $\aM$
that along $\tilde{\Sigma}$ lifts $\mu$.
 \end{proof}

\begin{remark}
  Note that in any dimension, and on any conformally compact manifold
  the formula~\nn{Q} defines a curvature quantity $Q^{g^\mu}$
  determined by the metric $g^\mu_\Sigma$ and the ambient geometry.
  (Note that the first part of the Proof of Theorem~\ref{tanein} does not
  use that $n$ is even, nor that $\g$ is Einstein.) By construction
  this satisfies an analogue of the celebrated Q-curvature conformal 
transformation law:
$$
Q^{e^{2\om}g^\mu}= Q^{g^\mu}+ \mathcal{P}_n\,  \omega,
$$
with $(-1)^n\big((n-1)!!\big)^2 \, \mathcal{P}_n= P_n$ defining $\mathcal{P}_n$.
When $n$ is even it is reasonable to view this as a (in general
non-intrinsic) generalisation of the Q-curvature.
\end{remark}

\begin{remark}
  The Q-curvature as discussed above is sometimes called the {\em
    critical Q-curvature} since it is associated to the dimension
  order GJMS operators. Related scalar curvature quantities associated
  to the other GJMS operators are sometimes called {\em non-citical
    Q-curvatures} and are also of interest for (for example) curvature
  prescription problems \cite{BrGoorigins}. In an obvious adaption of
  the ideas from Theorem \ref{Qthm}, holographic formulae for these
  arise (at least up to the orders covered by Theorem \ref{tanein}) by
  applying the operators $ \Big(-\frac{1}{I^2} I\cdot D\Big)^k$ to a
  suitable power of $\mu$. We leave the details to the reader.
\end{remark}

\section{The extension problems and their asymptotics} \label{ext}

Henceforth we consider a structure $(M,c,\si)$ with $\si$ a defining
density for a hypersurface $\Sigma$ and $I^2$ nowhere zero.  We
consider the problem of solving, off  $\Sigma$ asymptotically,
$$
I\cdot D
f=0\, ,
$$ 
for $f\in \Gamma \ce^\Phi[w_0]$ and some given weight $w_0$. For
simplicity we henceforth calculate on the side of $\Sigma$ where $\si$ is
non-negative, so effectively this amounts to working locally along the boundary
of a conformally compact manifold. 

\subsection{Solutions of the first kind}\label{first}
Here we treat an obvious Dirichlet-like problem where we view
$f|_\Sigma$ as the initial data.  Suppose that $f_0$ is an arbitrary
smooth extension of $f|_\Sigma$ to a section of $\ce^\Phi[w_0]$ over
$M$. We seek to solve the following problem:

\begin{problem}\label{extp}
  Given $f|_{\Sigma}$, and an arbitrary extension $f_0$ of this to
  $\ce^\Phi[w_0]$ over~$M$, find $f_i\in \ce^{\Phi}[w_0-i]$ (over
  $M$), $i=1,2,\cdots$, so that
$$
f^{(\ell)}: = f_0 + \si f_1  + \si^2f_2 +\cdots + O(\si^{\ell+1})
$$ solves $I\cdot D f = O(\si^\ell)$, off $\Sigma$, 
for $\ell\in \mathbb{N}\cup \infty$
as high as possible.
\end{problem}
\begin{remark}
  For $i\geq 1$ we do not assume that the $f_i$ are necessarily
  non-vanishing along~$\Sigma$. 
\end{remark}

We write $h_0 =d+2w_0$ so that $h f_0=h_0 f_0$, for example.  The
existence or not of a solution at generic weights is governed by the
following result.  
\begin{lemma}\label{gjmsstyle}
  Let $f^{(\ell)}$ be a solution of Problem~\ref{extp} to order
  $\ell\in \mathbb{Z}_{\geq 0}$. Then provided $\ell\neq h_0-2= n+2w_0-1$ 
  there is an extension 
$$
f^{(\ell+1)}= f^{(\ell)}+\si^{\ell+1}
  f_{\ell+1},
$$
 unique modulo $\si^{\ell+2}$, which solves
$$
I\cdot D f^{(\ell+1)}= 0 \quad \mbox{modulo} \quad O(\si^{\ell+1}).
$$  

If $\ell = h_0-2$ then the extension is obstructed by $P_{\ell+1} f_0|_\Sigma$. 
\end{lemma}
\begin{proof} Note that $I\cdot D f=0$ is equivalent to
  $-\frac{1}{I^2} I\cdot D f=0$ and so we can recast this as a formal
  problem using the Lie algebra $<x,y,h>$ from
  Proposition~\ref{slprop}. 
Using the notation from there
$$
y f^{(\ell+1)} = y f^{(\ell)} - 
x^\ell (\ell+1)(h+\ell)f_{\ell+1}+O(x^{\ell+1}). 
$$
Now $h f_{\ell+1}=\big(h_0-2(\ell+1)\big) f_{\ell+1}$, thus 
\begin{equation}\label{key1}
y f^{(\ell+1)} = y f^{(\ell)} - 
x^\ell (\ell+1)(h_0-\ell-2)f_{\ell+1}+O(x^{\ell+1}). 
\end{equation}
By assumption $y f^{(\ell)}=O(x^\ell)$, thus if $\ell\neq h_0-2$
we can solve $ y f^{(\ell+1)} = O(x^{\ell+1})$ and this uniquely
determines $f_{\ell+1}|_\Sigma$.

On the other hand if $\ell= h_0-2$ then~\nn{key1} shows that, 
modulo $O(x^{\ell+1})$, 
$$
y f^{(\ell)}= y \big( f^{(\ell)} + x^{\ell+1}f_{\ell+1} \big),
$$
regardless of $f_{\ell+1}$. It follows that the map $f_0\mapsto x^{-\ell} y
f^{(\ell)}$ is tangential and $ x^{-\ell} y f^{(\ell)}|_\Sigma$
is the obstruction to solving $ y f^{(\ell+1)} = O(x^{\ell+1})$.
By a simple induction this is seen to be a non-zero multiple of 
$y^{\ell+1}f_0|_\Sigma$.
\end{proof}

Thus by induction we conclude the following.
\begin{proposition}\label{gformal}
  For $h_0\notin \mathbb{Z}_{\geq 2}$ Problem~\nn{extp} can be solved
  to order $\ell$~$=$~$\infty$. For $h_0\in \mathbb{Z}_{\geq 2}$ the solution is
  obstructed by $[P_{h_0-1} f]|_\Sigma$; if, for a particular $f$,
  $[P_{h_0-1} f]|_\Sigma$ $=$~$0$ then there is a family of solutions to
  order $\ell=\infty$ parametrised by sections $f_{h_0-1}\in \Gamma
  \ce^\Phi[-d-w_0+1]|_\Sigma$.

If $(M,c,\si)$ is almost Einstein then Problem~\nn{extp} can be solved
  to order $\ell=\infty$ for all $(h_0+1)\notin 2\mathbb{Z}_{\geq 2}$.
\end{proposition}
Note the second part of the Proposition follows from the first using Theorem 
~\ref{tanein}.

\subsection{The formal solution operator} \label{fs}
Now to identify the asymptotics found above we re-examine the
extension problem~\ref{extp} and show first that, via the
$\frak{sl}(2)$ used above, in the generic case it is captured by a
simple ODE problem. This enables us to relate the asymptotics to 
Bessel functions.

We seek to define a formal operator $\colon K\colon$
which will map the ``initial data'' $f_0|_\Sigma$ to the solution. 
Thus as a first step we posit
\begin{equation}\label{solop}
f=(1+\alpha_1 \, xy + \alpha_2 \, x^2 y^2 +\cdots )f_0 
:=\ \colon K(z) \colon \, f_0\, ,
\end{equation}
where $z=xy$, $K(z)=1+\alpha_1 z + \alpha_2  z^2+\cdots \in \mathbb C[[z]]$ 
and  $\colon\: \ \colon$ indicates a {\em normal ordering}  defined by
$$
\colon z^k \colon = x^k y^k\, .
$$ 
Using $[y,x^k]=- x^{k-1} \, k(h+k-1)$ 
we have
$$
y\  \colon z^k \colon \, f_0 = \colon \Big[\Big(\frac{d}{dz}z\frac{d}{dz}-(h_0-1)\frac{d}{dz}+1\Big)z^k\Big]\colon
\  y f_0\, ,
$$
and so 
\begin{equation}\label{formsol}
y \, \colon K(z) \colon \, f_0 = \colon \big(z K''(z) - 
(h_0-2) K'(z) + K(z)\big) \colon \, y \, f_0  \, .
\end{equation}

Thus finding $K(z)$, so that $f:=\colon K(z) \colon
f_0$ formally solves $I\cdot D f=0$ for any $f_0$, amounts to finding
in $\mathbb C[[z]]$ 
a solution to the ordinary differential equation
\begin{equation}\label{Beq}
z K'' - (h_0-2) K' + K = 0\, .
\end{equation}
Formally solving this is equivalent to formally solving $z^2 K'' -
(h_0-2) zK' + z K = 0$ which can be re-expressed as
\begin{equation}\label{eulerf}
\big(E(E-h_0+1)+z\big) K(z) =0\, ,
\end{equation}
with 
$$
E=z\frac{d}{dz}\, ,
$$
the Euler operator.  Examining this applied to the terms $\alpha_{k-1}
z^{k-1}+\alpha_k z^k$ in the expansion of $K(z)$ we immediately recover the 
equivalent recursion relation
\begin{equation}\label{rec}
k(k-h_0+1)\alpha_k + \alpha_{k-1}=0\, ,
\end{equation}
which is solvable for $\alpha_k$, $k\in \mathbb{Z}_{\geq 1}$, so long
as $h_0\neq k+1=2,3,\ldots$.

In fact up to a simple change of variables and rescaling, the display~\nn{Beq} is Bessel's
equation. Thus, away
from the special weights $h_0=2,3,\ldots$ we have\footnote{The Bessel
  function of the first kind $J_\nu(u)$ is defined for complex values
  $u$ with $|\!\arg u|$~$<$~$\pi$. For the geometries we consider here,
  $\sqrt{z}$ --- being related to $\sigma$ and $\frac{-1}{I^2}I\cdot
  D$ --- can be purely imaginary, in which case the Bessel function
  may be rewritten in terms of a modified Bessel function.  For
  simplicity, we work with standard Bessel functions.  }
\begin{equation}\label{Kpower}
K(z) = z^{\frac{h_0-1}{2}} \Gamma(2-h_0)\, J_{1-h_0}\big(2\sqrt{z}\, \big)\, .
\end{equation}
The Gamma function is included to achieve the correct normalisation,
and the Bessel function of the first kind was chosen as its Taylor series 
expansion\footnote{The Pochhammer symbol
  $\big(k\big)_l=k(k-1)\cdots(k-l+1)$ and $(k)_0=1$.}
$$
z^{\frac{h_0-1}{2}}\Gamma(2-h_0) J_{1-h_0}\big(2\sqrt{z}\big)=
\sum_{m=0}^\infty \frac{(-z)^m}{m!\big(m+1-h_0\big)_m}$$ $$\hspace{5cm}=1+\frac{z}{h_0-2}+\frac{z^2}{2(h_0-2)(h_0-3)}+\cdots \, ,
$$
matches the series solutions found above.
 
Orchestrating the above results we have the following answer to
Problem~\ref{extp}, away from  the exceptional weights.
\begin{proposition}\label{fsprop} On $\ce^\Phi[w_0]$, $w_0\notin \{ \frac{j-d}{2}\mid j\in \mathbb{Z}_{\geq 2}\}$, the formal solution operator $\colon K\colon$ is given by 
 $$
f_0\stackrel{{\small \colon} K{\small \colon} \ }{\longmapsto} \sum_{m=0}^{\infty}\sigma^m  \, 
\frac{1}{m!\big(m+1-h_0\big)_m\big (I^2\big)^m} \, \big(I\cdot D\big)^m \, 
f_0\, .
$$
\end{proposition}

\vspace{.1cm}
\begin{remark}
  Readers familiar with~\cite{GrZ} and related treatments of
  Poincar\'e--Einstein asymptotics may recall an evenness property of
  the series expansions there.  This feature appears as a consequence
  of the form of our solutions (see, for example,
  Proposition~\ref{fsprop}), where every appearance of the
  scale~$\sigma$ is accompanied by that of the operator $I\cdot D$,
  both of which, upon making choices corresponding to the setting
  of~\cite{GrZ}, yields one power of the coordinate around which their
  expansion is based.
\end{remark}
\begin{remark}
  Curiously enough, in a different realisation of the ${\frak sl}(2)$
  (as well as for higher rank ${\frak g}$ and super Lie algebras,
  where $xy$ is replaced by products of generators $x_\alpha y_\alpha$
  summed over positive roots $\alpha$), the same normal ordered
  expression as $\colon K \colon$ (albeit at certain fixed values of
  $h_0$) appeared in the completely different contexts of higher spin
  systems and minisuperspace quantizations of supersymmetric black
  holes~\cite{Cherney:2010xh,Cherney:2009md,Cherney:2009mf,Cherney:2009vg,Sagnotti}.
  There the underlying representation of ${\frak g}$ is in terms of
  algebraic operations on tensors. The normal ordered Bessel function
  $\colon K\colon$ appeared, not as a solution generating operator,
  but instead was used to construct the differential in a complex
  whose cohomology described either higher spin excitations or states
  in a quantized black hole.
\end{remark}
\vspace{.1cm}

At this point it is not yet evident how the formal solution $\colon K(z)
\colon f_0$ depends on the choice of $f_0\in \ce^\Phi[w_0]$ extending
$f_0|_\Sigma$. This is settled as follows.

\begin{proposition}\label{Kx}
  The solution $\colon K(z) \colon f_0$ is independent of how $f_0$   
smoothly extends  $f|_\Sigma$ off $\Sigma$, as a section of $\ce^{\Phi}[w_0]$.
\end{proposition}
\begin{proof}
  Let $f_1\in
  \Gamma\ce^{\Phi}[w_0-1]$. Then $xf_1\in \Gamma\ce^{\Phi}[w_0]$ and
  we consider
$$
\colon z^k \colon \, x f_1 = x^k y^k x f_1.
$$
Using $[y^k,x]=-y^{k-1}
k(h-k+1) $ from~\nn{pid} we see that 
$$
x^k y^k x f_1 = -x \, \big(  k(h_0-k-1)x^{k-1}y^{k-1} - x^ky^k \big)f_1
= x\colon  \big( k(k+1-h_0)z^{k-1}+z^k \big) \colon \, f_1. 
$$
Thus 
$$
\colon K(z) \colon \, x \, f_1 =  x \, \colon \big(z K'' - 
(h_0-2) K' + K\big) \colon \, f_1 =  x \, \colon \Big[\, \frac{1}{z} \big(E(E-h_0+1)+z\big) K(z)\Big]  \colon \, f_1\, ,
$$
and we note the equation~\nn{Beq} has again arisen. Since 
$K(z)$ is a solution of this, the right-hand-side of the display vanishes.
\end{proof}

In the following it will be convenient to annotate $K(z)$ with $h_0$
so that $K^{h_0}(z)$ indicates the series in $\mathbb{C}[[z]]$ giving
the solution operator on $\ce^{\Phi}[w_0]$ (where $h_0=d+2w_0$ as
usual).

\subsection{Solutions of the second kind} \label{sec}
The equation~\nn{Beq}, being second order, admits a second independent
homogeneous solution. This corresponds to another solution of
the~$I\cdot D$ equation above, but to see this rigorously we first 
  generalise ({\it cf.} Problem~\ref{extp}) our
notion of solution as follows.
\begin{problem}\label{extp2}
  Given $\f_0|_{\Sigma}\in \Gamma \ce^\Phi[w_0-\alpha]|_\Sigma$ and an
  arbitrary extension $\f_0$ of this to $\Gamma\ce^\Phi[w_0-\alpha]$ over
  $M$, find $\f_i\in \ce^{\Phi}[w_0-\alpha -i]$ (over $M$),
  $i=1,2,\cdots$, so that
\begin{equation}\label{Jf}
\f: = \si^\alpha \big( \f_0 + \si\,  \f_1  + \si^2\, \f_2 +\cdots + O(\si^{\ell+1}) \big)
\end{equation}
solves $I\cdot D \f  = O(\si^{\ell+\alpha})$, off $\Sigma$, 
for $\ell\in \mathbb{N}\cup \infty$
as high as possible.
\end{problem}
\noindent As usual we shall write $I\cdot D \f =0 $ as a shorthand for the
statement that $\f$ is an $\ell=\infty$ solution.

The exponent $\alpha$, if not integral, takes Problem~\ref{extp2}
outside the realm of the universal enveloping algebra $\cU(\frak{g})$ and its
modules. However $\si=x\in \frak{g}$ is a smooth section of an oriented real
line bundle and as such $x^\alpha $ is well defined for any 
complex $\alpha$. What is more, as  follows  from  Remark~\ref{remarb},
within the operator algebra Corollary~\ref{corpid} allows the extension
\begin{equation}\label{pidext}
[x^\alpha,y]=x^{\alpha-1} \alpha (h +\alpha-1).
\end{equation}

Since $\f_0$ may be non-vanishing along $\Sigma$ it follows immediately
from~\nn{pidext} that $I\cdot D \f =0 $ has no solution unless 
$$
\alpha=0 \quad \mbox{or} \quad \alpha= h_0-1,
$$
where $h \f=h_0 \f$. Since $\alpha =0$ is the case of the
extension problem treated above, we set here $\alpha:= h_0-1$.
\newcommand{\oh}{\overline{h}}
Let us write $\oh_0$ for the $h$-weight of $\f_0$.
Observe that since $h (x^\alpha \f_0)=h_0 (x^\alpha \f_0)$, 
it follows that 
$$
1-\oh_0=h_0-1.
$$

Now write
$$
G(z)= 1+ \beta_1 z + \beta_2 z^2+\cdots  
$$
for the element of $\mathbb{C}[[z]]$ so that 
$$
 \colon G(z) \colon  = 1+\beta_1 \, xy + \beta_2 \, x^2 y^2 +\cdots  
,
$$
and $(x^{\alpha} \colon G(z) \colon ): \ce^\Phi[w_0-\alpha]\to
\ce^\Phi[w_0] $ is the putative solution operator (with $\colon\ \:
\colon$ the normal ordering as before).  Following the ideas  above
we now seek an ODE that determines $G$. But note that since $\alpha=h_0-1$
then, on $\ce^\Phi[w_0-\alpha]$, we have  
$y \, x^\alpha \colon G(z) \colon \, = x^\alpha \, y \, \colon G(z) \colon \ $.
Thus, upon the replacement of $h_0$ with $\oh_0$, we are in fact
reduced to {\em the same} ODE
\begin{equation}\label{same}
(z G')' - (\oh_0-1) G' + G =0 \quad \Leftrightarrow \quad \big(
E(E-\oh_0+1) +z \big)G(z) =0,
\end{equation}
as in~\nn{Beq}. It follows that $G(z)$ is given by~\nn{Kpower} with this
replacement, that is 
\begin{equation}\label{Koh}
G(z)= K^{\oh_0}(z)= z^{\frac{\oh_0-1}{2}} \Gamma(2-\oh_0)\, J_{1-\oh_0}\big(2\sqrt{z} \big)\, =
z^{\frac{1-h_0}{2}} \Gamma(h_0)\, J_{h_0-1}\big(2\sqrt{z} \big)\, .
\end{equation}
\begin{remark}
  Using $Ez^\alpha=z^{\alpha}(E+\alpha)$, It is easily verified that $z^\alpha G(z)$ is a (``second'')
  solution to~\nn{Beq}.
\end{remark}

For the record we note that, in terms of the weight $h_0$, the ODE~\nn{same} is 
\begin{equation}\label{ODE2}
(z G')' + 
(h_0-1) G' + G =0.
\end{equation}
Expressed alternatively, we wish to solve
$$
\big( E(E+h_0-1) +z \big)G(z) =0,
$$
which leads to the equivalent recursion relation 
$$
k(k+h_0-1)\beta_k + \beta_{k-1}=0\, .
$$
This is solvable for $\beta_k$, $k\in \mathbb{Z}_{\geq 1}$, now with
the restriction $h_0\neq 1-k= 0,-1, \ldots$.

Putting together the results from Section~\ref{fs} and Section
~\ref{sec} we summarise as follows.
\begin{proposition}\label{twosol}
Problem~\ref{extp2} has a solution only if $\alpha =0$ or $\alpha =h_0-1$.
If $\alpha =0$ and $h_0\notin \mathbb{Z}_{\geq 2}$
then an infinite order solution is given by 
\begin{equation}\label{one}
f_0\mapsto \, \colon K^{h_0}\colon \, f_0\, , 
\end{equation}
where $f_0$ is any section of $\ce^\Phi[w_0]$ smoothly extending $f_0|_\Sigma \in
\ce^\Phi[w_0]|_\Sigma$.

If $\alpha = h_0-1 $ and $h_0\notin \mathbb{Z}_{\leq 0}$ then 
an infinite order solution is given by
\begin{equation}\label{two}
\f_0 \mapsto x^{h_0-1} \colon K^{\overline{h}_0}\colon\,  \f_0 \, ,
\end{equation}
where $\f_0$ is any smooth extension of  $\f_0|_\Sigma$ to a section of
$\ce^\Phi[-d-w_0+1]$.

In each case the solution is independent of how $f_0$ (respectively
$\f_0$) extends $f|_\Sigma$ (respectively $\f|_\Sigma$) off~$\, \Sigma$. Thus~\nn{one} and~\nn{two} determine well-defined solution
operators
$$
(\colon K^{h_0}\colon) : \Gamma \ce^\Phi[w_0]|_\Sigma \to \Gamma \ce^\Phi[w_0]\, ,
$$
and 
$$
(x^{h_0-1}\, \colon K^{\overline{h}_0}\colon):\Gamma \ce^\Phi[w_0-h_0+1]|_\Sigma\to
\Gamma \ce^\Phi[w_0]\,Ê.
$$

For $h_0\notin \mathbb{Z}$ the solutions~\nn{one} and
\nn{two} are distinct if $f_0|_\Sigma$ and $\f_0|_\Sigma$ are not zero.
Further any formal solution $f\in
\Gamma\ce^\Phi[w_0]$ of the equation $I\cdot D f=0$ (off $\Sigma$)
is a linear combination of solutions of the form~\nn{one} and
\nn{two}.
\end{proposition}

\begin{remark}
When $h_0$ is a positive integer the second solution~\nn{two}
still provides an infinite order solution to Problem~\ref{extp2}.
Moreover, this is an instance of the situation where one can
compose tangential operators:
for positive integer~$h_0$, the second solution also furnishes a solution to a modified
version of Problem~\ref{extp2} 
\begin{equation*}
f_0 \mapsto x^{h_0-1} \colon K^{\overline{h}_0}\colon\, y^{h_0-1}  f_0 \, ,
\end{equation*}
where instead of $\f_0|_\Sigma$, one is given an $f_0|_{\Sigma}\in \Gamma \ce^\Phi[w_0]|_\Sigma$
and an
  arbitrary extension $f_0$ of this to $\Gamma\ce^\Phi[w_0]$ over
  $M$. Again, this solution is independent of how $f_0$ extends $f_0|_\Sigma$  off~$\, \Sigma$.
\end{remark}

Finally, note that $h_0=1$ ({\it i.e.} $w_0=-\frac{d-1}2$) is allowed
for both~\nn{one} and~\nn{two}, but in this case they coincide. This
amounts to the situation of a repeated root of the indicial equation
for the Frobenius method for second order ODEs. This suggests a new
second solution, perhaps of the form $ \F (\si) \log \si $, where $\F (\si)$ is
again a power series type solution. The weights $h_0\in
\mathbb{Z}\setminus \{1\}$ should also be dealt with by log
terms. With $\si$ a section of~$\ce_+[1]$ we can make sense of $\log
\si$ using the notion of log-densities developed in Section
~\ref{dlogd}.

\subsection{Solutions with log terms}\label{logs}
We now consider the case of weights $h_0$~$\in$~$\mathbb{Z}$. From above we
see that for each such weight there is just one smooth solution,
{\it i.e.} solution to Problem~\ref{extp2}. In view of the symmetry between
the cases of weights $h_0$ and $\oh_0$ it suffices to treat the
case of $h_0\in \mathbb{Z}_{\geq 1}$.

We shall assume that $(M,c,\si) $ is conformally compact with $\si
\geq 0$.  Recalling the discussion of log density bundles from Section
~\ref{dlogd} we note that $\log \si$ is then well defined on the
interior of $M$, where $\si>0$, and we include $\log \si =\log x$ into
our formal computations. Since we understand pointwise the meaning of
complex linear operations on log densities we can take the tensor product of
$\cf[w]$ with vector bundles, and in particular weighted tractor bundles.

A slight subtlety arises at this point, because $\log \si$ is not an
eigen-section for the weight operator, {\it cf.}~\nn{law of the logs}.
Thus we introduce a second scale: let~$\tau$ denote some true scale on
$M$, {\it i.e.}\ $\tau\in \Gamma\ce[1]$ is nowhere zero on $M$; in
particular $\tau|_\Sigma$ gives a scale on $\Sigma$, and we may write
$g^\tau$ for the metric $\tau^{-2}\bg$ on $M$ determined by $\tau$.

\begin{problem}\label{extp3} 
  Let $h_0=d+2 w_0$ be in $\mathbb{Z}_{\geq 2}$.
  Given non-zero $f_0|_{\Sigma}\in
  \Gamma \ce^\Phi[w_0]|_\Sigma$ and an arbitrary extension~$f_0$ of this to
  $\Gamma\ce^\Phi[w_0]$ over $M$, find $f_i\in \ce^{\Phi}[w_0 -i]$,
  $i=1,2,\cdots$, and $\f_j\in \ce^{\Phi}[-d-w_0+1 -j]$ (over $M$),
  $j=0,1,\cdots$, so that
\begin{equation}\label{logf}
  f: = \big(f_0 + \si f_1   +\cdots \big) + \si^{h_0-1} (\log \si-\log \tau) \big(\f_0+\si \f_1 + \cdots  \big)+ O(\si^{\ell+1}) 
\end{equation}
solves $I\cdot D f = O(\si^\ell)$, off $\Sigma$, for $\ell\in
\mathbb{N}\cup \infty$ as high as possible. More precisely, if
$\ell\geq h_0-1$ then the display should also include $+\, O(\si^{\ell+1}
\log \si)$, meaning up to the addition of a formally smooth term
tensor product with $ \si^{\ell+1} \log \si $, and then we seek to
solve $I\cdot D f = O(\si^\ell)+ O(\si^\ell \log \si)$, for $\ell$ as
high as possible.

If $h_0=1$ we seek instead a solution with $\f_0|_\Sigma$ not zero and
taken as the initial data, and $f_0=0$.
\end{problem}

\begin{remark}
  Three comments on the statement of Problem~\ref{extp3} are important:\\[2mm]
  First note that we have explicitly entered a second scale, {\it
    viz.}\ the scale $\tau$ that extends to the boundary. Although we
  may treat the $I\cdot D$ equation without doing this, the man\oe uvre
  is necessary if we require that $f$ is an eigenfunction of the weight operator
  (enabling the interior relation~\nn{IdotDsc}). Note that $ \log \si
  -\log \tau =\log r$ is a function.\\[2mm]
  Second we have decreed where the log terms enter by adding a term of
  the form $ \si^{h_0-1}$ $\log \si$  $\overline{\mathsf
    f}$. Alternatively we could have added $ \si^{\beta} \log \si \
  \overline{\mathsf f}$ for some $\beta$. It is straightforward to
  then show that $\beta=h_0-1$ is forced.
  \\[2mm]
  Third, in the statement of the problem we have stipulated
  that~$\log\tau$ appear in the combination $
  \log\sigma-\log\tau$. Although we cannot relax this requirement, we
  shall see that the simplest presentation of the solution is 
  not {\em explicitly} in this form, but differs by terms involving  various
  reorderings of the operators $y$ and $\log\tau$.
\end{remark}  
  
\begin{remark}\label{othersol}
  Allowing $f_0=0$, Problem~\ref{extp3} always has a solution of the
  second kind (discussed in Section~\ref{sec}) with the leading power
  of $\si$ equaling the integer $h_0-1$. It is given by the formula
  $x^{h_0-1}\colon K^{2-h_0}(z)\colon f_{h_0-1}$ where
  $f_{h_0-1}|_{\Sigma}\in \Gamma \ce^\Phi[-d-w_0+1]|_\Sigma$ and~$f_{h_0-1}$ 
  is an arbitrary extension of this to
  $\Gamma\ce^\Phi[-d-w_0+1]$ over $M$. In the following discussion, we
  focus on the case of vanishing $f_{h_0-1}$; this term will be
  supplanted by the leading coefficient of the logarithm $\f_0$ (which is of
  the same weight as $f_{h_0-1}$).
\end{remark}

\newcommand{\sff}{\mathsf f}
\newcommand{\ssf}{\hspace{.2mm}\overline{\mathsf f}}
\newcommand{\ssg}{\mathsf g}

Let us write 
$$
\, {\mathsf f}:= f_0\, + x f_1 \,  + x^2 f_2 \, +\cdots 
$$
and
$$
\ \  \overline{\mathsf f}:= \f_0+ x\,  \f_1 +x^2\,  \f_2+\cdots\, .
$$
Thus we seek a solution  of the form 
\begin{equation}\label{fs3}
\sff + x^{h_0-1} \, (\log x -\log \tau)\ \ssf ~.
\end{equation}
To determine the equations on $\sff$ and $\ssf$ we need to draw $\log
x$ into our operator algebra.  Formally, the identity
$$
\lim_{k\rightarrow 0}\frac{x^k-1}{k}=\log x
$$
coupled with the relations $[h,x]=2x$ and $[y,x^k]=-x^{k-1} k(h+k-1)$ from Proposition~\ref{slprop}
and Corollary~\ref{corpid} suggest that
\begin{equation}\label{ylogx}
[h,\log x]=2\, ,\qquad [y,\log x]=-x^{-1}(h-1)\, .
\end{equation}
The first of these relations does  hold rather generally; if $\mu\in \Gamma \ce_+[w_0]$ is any positive weighted conformal density,
then $\log\mu$ is a weighted log density, {\it viz.} a section of ${\mathcal F}[w_0]$ and, by virtue of~\nn{law of the logs}
obeys
$$
[h,\log\mu]=2w_0\, 
$$
read as an operator relation acting on any section of a weighted tractor
bundle.  The second relation in~\nn{ylogx} can be easily proven using ambient
tractor machinery, as provided by Proposition~\ref{toadd2}. It also
holds on arbitrary sections of weighted tractor bundles.  In fact both
relations in~\nn{ylogx} can be shown (again by ambient techniques) to
hold also on log densities or tensor products of these with conformal
densities, or more generally on a log density tensor product with any
weighted tractor bundle. In the remainder of this Section we will
develop various consequences of the above identities, with the same
range of validity, and state them without explicit reference to the
class of sections on which they act.

Now we apply $y$ to~\nn{fs3}.  Since the weight of $\ssf $ is such
that $[y,x^{h_0-1}]\ssf=0$, and similarly $[y,x^{h_0-1}] (\log x-\log
\tau) \ssf=0$ new terms arise only from the commutator $[y,\log
x-\log \tau ]$. Thus the equation from $y$ applied to~\nn{fs3} is
\begin{equation}\label{meq}
0=y \, \sff  
-  \big[1-h_0 \big] x^{h_0-2}\, \ssf  
+ x^{h_0-1} (\log \!\,x\ y-y \log \tau) \, \ssf ~.
\end{equation}
Some observations greatly simplify the analysis of this. First it is
not difficult to see that the $\log x$ terms must vanish separately and so 
$$
y \, \ssf=0.
$$
Given the weight of $\ssf$ we know that there is a formal power
series type solution of the form $\ssf = \colon \F(z)\colon \, \f_0$, so we
obtain the equation on $\F(z)$ from~\nn{ODE2}:
\begin{equation}\label{Besse}
\big[E(E+h_0-1)+z\big]\,  \F(z) =0\, .
\end{equation}
From Section~\ref{sec} this  is solvable for $h_0\in
\mathbb{Z}_{\geq 1}$, yielding 
\begin{equation}\label{FF}
\F(z)= K^{\oh_0}(z).
\end{equation}
At this point our proposed solution~\nn{fs3} reads $$\sff + x^{h_0-1} \, (\log x -\log \tau)\ \colon K^{\oh_0}(z) \colon \, \f_0$$
and we have as yet not determined $\sff$ or $\f_0$. Our second observation is that we may modify the expression~\nn{fs3}
and seek a formal solution of the form
\begin{equation}\label{fs4}
\sff + x^{h_0-1} \, \big(\log x \ \colon K^{\oh_0}(z)\colon\   -\ \colon K^{\oh_0}(z)\colon\, \log \tau\big)\ \f_0 ~,
\end{equation}
because moving $\log\tau$ past the operator $\colon
K^{\oh_0}(z)\colon$ produces terms, each of a well-defined weight,
which can be absorbed into the as yet undetermined~$\sff$ at orders
higher than~$x^{h_0-1}$.   The utility of this man\oe uvre will soon be
clear. 

Next we upgrade the identity~\nn{formsol} to an
operator relation for formal solution operators $\colon K(z)\colon\in
{\mathbb C}[[z]]$,
\begin{equation}\label{pretty snazzy identity}
y \, \colon K(z)\colon \ = \ \colon \big(z K''(z) + 
2 K'(z) + K(z)\big) \colon \, y - \colon K'(z)\colon \, y h\, .
\end{equation}
Using this, as well as~\nn{law of the logs} and the fact that~\nn{FF} solves~\nn{Besse}, we find that the equation 
arising from the 
modified expression for our proposed solution~\nn{fs4} is
\begin{equation}\label{inho}
y \, \sff  
=  \Big[1-h_0 \Big] x^{h_0-2}\, \colon K^{\oh_0}(z)\colon \, \f_0  - 2\, x^{h_0-1} \,  \colon \frac{dK^{\oh_0}(z)}{dz}\colon\, y \f_0~.
\end{equation}
Notice that for  any $\f_0$ all log terms have dropped out leaving an
inhomogeneous equation for $\sff$.
Settling this up to terms of order $x^{h_0-2}$ is the problem
\begin{equation}\label{key}
y \, \sff  
=  \Big[1-h_0 \Big] x^{h_0-2}\, \colon K^{\oh_0}(z)\colon \, \f_0 + O(x^{h_0-1}).
\end{equation}
Here we see that there are two cases $h_0=1$, and $h_0\in
\mathbb{Z}_{\geq 2}$. We treat the latter first.

For $h_0\in \mathbb{Z}_{\geq 2}$, and $f_0|_\Sigma$ non-zero, the
equation~\nn{key} can be treated in two stages.  The first of these
consists of solving $y \, \sff = O(x^{h_0-2})$. This is simply an
initial finite part of the Problem~\ref{extp} treated in Section
~\ref{first}. Thus we obtain a solution (the unique solution to the
given order, if $f_0|_\Sigma$ is assumed not zero) $\sff = \colon
F_{h_0-2}(z) \colon\, f_0 + O(x^{h_0-2})$ where the {\em polynomial}
$F_{h_0-2}(z)$, arising as the operator expressing the partial sum up of terms up to order
$x^{h_0-2}$ in $\sff$, obeys
$$
\big[E(E-h_0+1)+z\big] F_{h_0-2}(z)= O(z^{h_0-2})\, .
$$
Explicitly 
\begin{equation}\label{polynomial}
F_{h_0-2}(z)=1+\frac{z}{h_0-2}  + \frac{z^2}{2!(h_0-2)(h_0-3)}  +\ \cdots\ +\frac{z^{h_0-2}}{\big((h_0-2)!\big)^2} \ ,
\end{equation}
which amounts to solving~\nn{rec} when $k = 1,2,\ldots , h_0-2$.  So
\begin{equation}\label{spitsout}
y\,\colon F_{h_0-2}(z)\, \colon  f_0 
= \frac{1}{\big((h_0-2)!\big)^2}\   x^{h_0-2}y^{h_0-1} f_0 . 
\end{equation}
Now inspecting~\nn{rec} we see the following: the equation $\colon \,
\big[E(E-h_0+1)+z\big] F(z)\, \colon f_0 = 0 $ cannot be solved to
higher order in general, with the tangential operator $y^{h_0-1} f_0 =
P_{h_0-1}f_0$ (of Theorem~\ref{tanth}) the obstruction (if this
vanishes we may continue, {\it cf.}\ Lemma ~\ref{gjmsstyle} and Proposition
~\ref{gformal}).   At the next order $\f_0$ provides the rescue
in~\nn{key}, in that~\nn{key} can be solved by setting
\begin{equation}\label{fbarf}
\f_0=- \frac{1}{(h_0-1)!(h_0-2)!} \ y^{h_0-1} f_0 .
\end{equation}

Equation~\nn{inho} is now solved canonically to all higher orders as follows. We extend 
$\sff = \colon F_{h_0-2}(z)\colon \, f_0 + O(x^{h_0-1})$ by first setting the undetermined
coefficient of $x^{h_0-1}$ to zero. (Note that any other choice is tantamount to
adding to our solution a non-trivial solution of the second kind, as
discussed in Remark~\ref{othersol}.)  By~\nn{fbarf} $\f_0$ is
determined by $f_0$, whence so the right-hand-side of~\nn{inho} as a power series
type expansion in~$x$. 
Now in equation~\nn{inho} we have dealt with
all terms below order $x^{h_0-1}$. But after this the 
equation  is
solvable recursively, since it is simply an inhomogeneous version of the recursion relation~\nn{rec} of the form 
$k(k-h_0+1) \alpha_k +  \alpha_{k-1}= r_{k-1} \, ,\quad k= h_0,h_0+1,\cdots$,
 where $r_k$ indicates terms
corresponding to the right-hand-side of~\nn{inho}.

In fact, this last problem can be neatly formulated in terms of an
inhomogeneous ordinary differential equation along the lines
of~\nn{Beq}. This surprisingly enables a complete and  explicit solution to
the problem. In total the treatment involves several subtleties;
to reveal and deal with each of these we break the treatment into a
number of stages.  
The first step is  that we express also the $O(x^{h_0})$ terms in~$\sff$
in the solution generating operator language 
\begin{equation}\label{further decompose}
\sff = \colon F_{h_0-2}(z) \colon \, f_0 + x^{h_0} \, \colon \H (z)\colon \, y\f_0\, ,
\end{equation}
where it now only remains to determine $\H(z)\in{\mathbb C}[[z]]$ (here $\f_0$ is given as in~\nn{fbarf}).
Inserting this, along with~\nn{spitsout} and~\nn{fbarf} into our equation~\nn{inho} gives
$$
x^{h_0-2}\Big((1-h_0)+xy\, \colon z \H(z)\colon\Big)\, \f_0
= x^{h_0-2}\, \colon\Big((1-h_0)\,  K^{\oh_0}(z) -2\,   z\frac{dK^{\oh_0}(z)}{dz}\Big)\colon\, \f_0\, .
$$
Here we used Corollary~\nn{corpid} and the fact that $\f_0\in\ce^{\Phi}[-d-w_0+1]$, along with the simple operator  identity
$x\, \colon K(z)\colon \, y=\colon zK(z)\colon$ for any $K(z)\in{\mathbb C}[[z]]$.
We can solve this equation using the formal solution operator methods explained in Section~\ref{first} (and in particular 
equation~\nn{formsol}). This gives the inhomogeneous, linear ODE 
\begin{equation}\label{inhom ODE}
\big(E(E+h_0-1)+z\big)[z\H(z)]=-(2E+h_0-1)[K^{\oh_0}(z)-1]\, .
\end{equation}
Clearly the Frobenius method yields a unique formal solution for $\H(z)$,
the leading terms of which we explicate below
$$
\H(z)=\scalebox{.95}{$ \frac{h_0+1}{h_0^2}$}\  - \ \scalebox{.95}{$ \frac{(h_0+2)(3h_0+1)}{4h_0^2(h_0+1)^2}$}\, z
\ + \ \scalebox{.95}{$ \frac{(h_0+3)(11h_0^2+18h_0+4)}{36h_0^2(h_0+1)^2(h_0+2)^2}$}\, z^2 $$ $$\qquad\quad\ 
\ \ - \ \scalebox{.95}{$ \frac{(h_0+4)(25h_0^3+98h_0^2+99h_0+18)}{288h_0^2(h_0+1)^2(h_0+2)^2(h_0+3)^2}$}\, z^3
\ + \ \cdots\, .
$$
The solution to all orders for $\H(z)=\sum_{k=0}^\infty\gamma_kz^k \in {\mathbb C}[[z]]$
is given by the recurrence
\begin{eqnarray*}
\gamma_0&=&\frac{h_0+1}{h_0^2}\, ,\\
\gamma_k&=&-\frac{1}{(k+1)(h_0+k)}\left[\gamma_{k-1}+\frac{(-1)^{k+1}(h_0+2k+1)}{(k+1)!\big(h_0+k\big)_{k+1}}\right]\, .
\end{eqnarray*}
\begin{remark}
Use of the computer package Maple gives the closed-form solution to this recursion:
$$
\gamma_k=\frac{(-1)^k\Big(h_0\sum_{j=0}^{k-1}\frac{h_0+2j+3}{(j+2)(h_0+j+1)}+h_0+1\Big)(h_0-1)!}{h_0(h_0+k)!(k+1)!}\, .
$$
\end{remark}

Putting together the polynomial $F_{h_0}(z)$ and power series $\H(z)$, $K^{\oh_0}(z)$ $\in$  ${\mathbb C}[[z]]$ 
(given by equations~\nn{polynomial},~\nn{inhom ODE} and~\nn{Koh}, respectively) we may now build a solution 
operator~$\widehat\cO$. Acting on sections 
of $\ce^\Phi[w_0]$, this is given by
\begin{eqnarray}\label{hatO}
\widehat\cO&=&\colon F_{h_0-2}(z)\colon -\frac{\colon z^{h_0} \, \H(z)\colon}{(h_0-1)!(h_0-2)!}\nonumber\\
&-&\frac
{x^{h_0-1}\big(\log x \ \colon K^{\oh_0}(z)\colon\   -\ \colon K^{\oh_0}(z)\colon\, \log \tau\big)y^{h_0-1}}{(h_0-1)!(h_0-2)!} \, .
\end{eqnarray}
By construction this has the required property 
\begin{equation} \label{yhat}
y\, \widehat\cO f_0 = 0 
\end{equation}
for any section $f_0$ of $\ce^\Phi[w_0]$. However, we have not yet
reached the ideal solution to the problem~\nn{extp3}. The point here is 
that~\nn{hatO}  is not independent of how $f_0$ is extended off
$\Sigma$ from given boundary data~$f_0|_{\Sigma}\in \Gamma
\ce^\Phi[w_0]|_\Sigma$. This statement is readily verified by explicitly
computing $\widehat \cO x f_1$ for some $f_1\in\Gamma\ce^\Phi[w_0-1]$
 \begin{equation}\label{leftover}
   \widehat \cO\,  x f_1 =   \frac{x\, \colon z^{h_0-2} K^{\oh_0}(z)\colon\, f_1}{\big((h_0-2)!\big)^2} \, .
\end{equation}
The last computation
relied on the second relation in Corollary~\ref{corpid} as well as
analog of the identity~\nn{pretty snazzy identity}, adapted to the
case where the operator $x$ appears on the right, rather than $y$ on
the left, namely
\begin{equation}\label{right snazziness}
  \colon K(z)\colon \, x = x\, \colon \big(z K''(z)+K(z)\big)\colon-x\,
 \colon K'(z)\colon\, h\, .
\end{equation}
There is however, no analog of the second relation in~\nn{ylogx} for
the commutator of the operator $y$ with $\log x$, because $x$ and
$\log\tau$ simply commute. This absence precisely explains the
leftover terms in~\nn{leftover} whilst at same time suggests a possible remedy.
Examining the operator~$\widehat\cO$ given above, we see that there is
a complete symmetry between the operator $x$ on the left and $y$ on
the right (remembering that $\colon z^k\colon =x^ky^k$) except for
where $\log\tau$ appears in the place where symmetry would have
suggested a term ``$-\log y$''.  
Remarkably it is possible to adjust~\nn{hatO} in such a way that
$-\log\tau$ does effectively play the {\it r\^ole} of the logarithm
of~$y$ for the algebraic purposes required.  Remembering that we may
always add some amount of a solution of the second kind, to the one
obtained already in~\nn{hatO}, we find that gives us sufficient leeway
to perform such a man\oe uvre.  
Our cure begins by observing that for any section $\f_0$ of $\ce^\Phi[-d-w_0+1]$ (so that $h\f_0=(2-h_0)\f_0=\oh_0\f_0$) we have the identity
\begin{equation}\label{leftlog}
[y,x^{h_0-1}\log x]\, \f_0 = (1-h_0)\,  x^{h_0-2}\f_0\, ,
\end{equation}
which was obtained using Corollary~\ref{corpid} and the algebra of $y$
with $\log x$ in~\nn{ylogx}.  What we would like is a close analog of
this result for commutators of the operator $x$ with $\log\tau$ and
powers of $y$. While {\it a priori} the existence of such is not at all
guaranteed,  it turns out that the Weyl averaged operator
 \begin{equation}\label{WeylO}
 \big(\log\tau\,  y^{h_0-1}\big)_{\rm W} := \frac{\log\tau \, y^{h_0-1}+y^{h_0-1}\, \log\tau}2\, ,\qquad h_0\in {\mathbb Z}_{\geq 1}\, ,
 \end{equation}
obeys 
\begin{equation}\label{rightlog}
\Big[ \big(\log\tau \, y^{h_0-1}\big)_{\rm W},x\Big]\, f_1=(1-h_0)\,  y^{h_0-2}f_1\, ,
\end{equation}
where $f_1$ is any section of $\ce^\Phi[w_0-1]$ (so that
$hf_1=(h_0-2)f_1$).  The validity of the identity~\nn{leftlog} is
essentially the reason why the introduction of the logarithm of $x$
enabled us to write the solution operator $\widehat\cO$ in~\nn{hatO}
annihilated by the action of $y$ from the left.  The new
identity~\nn{rightlog} exactly mimics that result, for the case now of $x$
acting from the right. This leads us to the new solution
generating operator:
\begin{eqnarray}\label{O}
\cO&=&\colon F_{h_0-2}(z)\colon -\frac{\colon z^{h_0} \, \H(z)\colon}{(h_0-1)!(h_0-2)!}\nonumber\\
&-&\frac
{x^{h_0-1} \log x \ \colon K^{\oh_0}(z)\colon\  y^{h_0-1}\  -\ x^{h_0-1} \ \colon K^{\oh_0}(z)\colon\, \big(\log \tau\,  y^{h_0-1}\big)_{\rm W}}{(h_0-1)!(h_0-2)!} \, .
\end{eqnarray}
It is then a matter of straightforward algebra, using the
relations~\nn{leftlog}, or~\nn{rightlog}, along with~\nn{pretty snazzy
  identity}, or ~\nn{right snazziness}, to establish the critical properties
$$
y \, {\cO} f_0=0,
\quad \mbox{and} \quad {\cO} x f_1=0,
$$ 
for arbitrary sections $f_0$ and $f_1$ of
$\ce^\Phi[w_0]$ and $\ce^\Phi[w_0-1]$, respectively. In fact for the
case of verifying $y {\cO} f_0=0$, this also follows from~\nn{yhat},
since by construction $\widehat\cO f_0$ and ${\cO} f_0$ differ by a
solution of the second kind.  For the case of showing ${\cO} x f_1=0$
we follow the idea of Proposition ~\ref{Kx}, although the situation
here is significantly more delicate: It is straightforward to verify
that, through the judicious introduction of the Weyl averaged
expression~\nn{WeylO} in~\nn{O}, we are again led to the equation
\nn{inhom ODE} on $z\H(z)$; the equation which $\H(z)$ solves, by its
definition.  
At this juncture, we also note that an easy variant of the computation
directly above shows that replacing $\log\tau$ by $\log\tau\ +\ x\,
t_1$ where $t_1$ is any section of $\ce[-1]$ does not change the
solution~$\cO f_0$. In other words, the solution only depends on
initial data $\tau|_\Sigma$ and is independent of the choice of
extension of this to $\ce[1]$ over $M$.

\begin{remark}
  One might wonder whether, by making a distinguished choice of second
  scale $\tau$, it could be possible to avoid introducing the Weyl
  ordered combination of operators $y$ and $\log\tau$ defined
  in~\nn{WeylO}. In fact it is straightforward to 
 prove that $y$ and
  $\log\tau$ do not commute when $I^2$ is non-vanishing.
\end{remark}

Now for the case $h_0=1$: We again search for a formal solution of the
form~\nn{fs4} with $h_0=1$, but recall in this instance we take
$\f_0|_\Sigma$ as the given initial data. This is required
by~\nn{inho} since the first term on the right-hand-side now
vanishes. We determine directly $\ssf$ in terms of $\f_0$
using~\nn{FF}. Then we write $\sff$ as in~\nn{further decompose} but
without the first term since~$f_0=0$ here. Thus the expansion
of~$\sff$ is given in terms of $\H(z)$, which (as for higher $h_0$
cases) is determined by the inhomogeneous ODE~\nn{inhom ODE}. To ensure that 
the solution is independent of how $\f_0|_\Sigma$ is extended off~$\Sigma$ we
use the same Weyl averaging technique as above. Thus, in
this case, the solution generating operator is given by the numerators
of the second and third terms in~\nn{O}, evaluated at $h_0=1$.

We have by now established the following
\begin{theorem}\label{logsol}
  Problem~\ref{extp3} has a canonical infinite order solution with
  $f_{h_0-1}=0$ when $1\neq h_0\in {\mathbb Z}_{\geq 2}$
\begin{equation}\label{punchl}
(f_0,\tau)\mapsto \cO f_0\, ,
\end{equation}
where $f_0$ and $\tau$ are, respectively, any sections of  $\ce^\Phi[w_0]$ and $\ce[1]$ (the solution  operator~$\cO$ is given explicitly in~\nn{O}).

If $h_0=1$, then an infinite order solution is given by
\begin{equation}\label{punchl1}
\big(\f_0,\tau\big)\mapsto \ocO \, \f_0\, ,
\end{equation}
where $\f_0$  and $\tau$ are, respectively, any sections of  $\ce^\Phi\big[\frac{1-d}{2}\big]$ and $\ce[1]$.
The solution  operator~$\ocO$ is
\begin{equation*}
\ocO\ =\ 
\log x \ \colon J_0\big(2\sqrt{z}\big)\colon\   -\ \colon J_0\big(2\sqrt{z}\big)\colon\, \log \tau+ \colon z \, \H (z)\colon
 \ .
\end{equation*}
\vspace{-2mm}

In each case the solution $\cO f_0$ (respectively $\overline\cO\,  \f_0$) is independent of how $\tau$ as well as $f_0$ (respectively
$\f_0$) extend $\tau|_\Sigma$ and $f|_\Sigma$ (respectively $\f|_\Sigma$) off~$\, \Sigma$ and thus determine well-defined solution
operators
$$
\cO : \big(\Gamma \ce^\Phi[w_0]|_\Sigma,\Gamma \ce[1]|_\Sigma\big) \to \Gamma \ce^\Phi[w_0]\, ,
$$
and 
$$
\ocO:\Big(\Gamma \ce^\Phi\big[\frac{1-d}{2}\big]|_\Sigma, \Gamma \ce[1]|_\Sigma\Big)\to
\Gamma \ce^\Phi\big[\frac{1-d}{2}\big]\,Ê.
$$
\end{theorem}

\begin{remark} \label{PEcase}
The case of $(M,c,\si)$ being Poincar\'e-Einstein.
For Problem ~\ref{extp3} we found that the coefficient of $\log \si$ satisfies 
$$
\f_0= - \frac{1}{(h_0-1)!(h_0-2)!} P_{h_0-1} f_0
$$
where $P_k$ is the
tangential operator from Theorem ~\ref{tanth}.  Any other solution of
the $I\cdot D f=0 $ equation, is a linear combination of this with a
solution of Problem ~\ref{extp2}.  

It is interesting to note the specialisations that arise if
$(M,c,\si)$ is a Poincar\'e-Einstein structure, which we now assume here.
Then the scale tractor $I$ is parallel, and the interior is
Einstein. This means that on the interior $(M_+,c)$ the powers of
$I\cdot D$ give the {\em interior GJMS operators}, when acting on densities
of the appropriate weight. This features precisely as follows.

Suppose that $j\in \mathbb{Z}_{\geq 1}$ and $h_0= 2j $ and so $f_0$
has conformal weight $w_0= j-\frac{d}{2}$.  Using that $(M,c,\si)$ is
conformally Einstein, from~\cite{powerslap} we have
that $y^j f_0 = x^j Y_j f_0$ where (up to a constant) $Y_j$ is the
order $2j$ GJMS operator (in general in fact the generalisation
thereof to a tractor twisted version) on $(M,c)$. Thus $y^{h_0-1} f_0
= y^{j-1}x^jY_j f_0$ and using~\nn{pid} this is $O(x)$. Thus by
~\nn{fbarf} $\f_0|_\Sigma=0$ and by Proposition~\ref{Kx}, also $\colon
K^{\oh_0} \colon \, \f_0 =0$. Thus the canonical solution has no log terms.

In summary, if $(M,\g)$ is Poincar\'e--Einstein then: 
\begin{itemize}

\item If $h_0=3,5,7,\ldots$,  then $P_{h_0-1}$ is the 
(tractor twisted) GJMS operator. 
\vspace{1mm}

\item If $h_0=2,4,6,\ldots$, then $P_{h_0-1}$ is trivial, $\f_0=0$ and there are no 
log  terms in the canonical solution. 

\end{itemize}

\end{remark}

Finally recall that we introduced the scale choice $\tau$ so that we
could describe a solution with a well-defined conformal weight,
despite the presence of log terms. Having solved the problem it is
most natural to finally express the answer in the scale determined by~$\tau$, 
that is in terms of the metric $g^\tau:=\tau^{-1}\bg$. In this
trivialisation of the weight bundles, $\tau$~itself is represented by
the constant function 1. Then since $\log 1=0$ the expression, in this
scale, for the solution is obtained from that found above by
simply formally 
replacing~$\si$ and its powers by $r=\si/\tau$ and its powers.

\subsection{The log density problem}\label{logdpr}
We now consider the problem of formally solving 
$$
I\cdot D\,  U = 0
$$
for a log density $U\in\Gamma\cf[1]$. We again treat a Dirichlet type problem
where~$U|_\Sigma$ is the initial data. Given an arbitrary smooth extension $U_0\in\Gamma\cf[1]$
of~$U|_\Sigma$ we study the following problem:
\begin{problem}\label{logprob}
Given $U_0|_\Sigma$ and an arbitrary extension of this to $\cf[1]$ over $M$, find $U_i\in\ce[-i]$ (over $M$, $i\in {\mathbb Z}_{\geq1}$), such that
$$
U^{(l)}:=U_0 + \si \, U_1+\si^2\,  U_2 +\cdots+O(\sigma^{l+1})
$$
solves $I\cdot D\,  U=O(\sigma^{l})$, off $\Sigma$, for $l\in{\mathbb N}\cup\infty$ as high as possible.
\end{problem}
\begin{remark}
 Since the sum of a log density and weight~0 conformal density ({\it i.e.}, a function) is again a log density, all terms in the expansion $U_i$ with $i\geq1$ are taken to be conformal densities with weight such that the product $\sigma^i U_i$ is a weight zero 
 conformal density. Moreover, since $I\cdot D \, U_0$ is a weight~$-1$ conformal density (even though $U_0$ is a log density), the solution to the above
 problem can be achieved by applying the results developed in previous sections for conformal densities, in particular those for $h_0=d$ (because then $w_0=0$).
\end{remark}

In light of the above remark, from Theorem~\ref{logsol} we directly obtain the following

\begin{corollary}
Problem~\ref{logprob} has a canonical infinite order solution with
$$
(U_0,\tau)\mapsto \cO|_{h_0=d}\,  f_0\, ,
$$
where $U_0$ and $\tau$ are, respectively, any sections of  ${\mathcal F}[1]$ and $\ce[1]$ and the solution  operator $\cO|_{h_0=d}$ is 
precisely the one given in~\nn{O} save that $h_0$ is set to the value $h_0=d$.

The solution $\cO|_{h_0=d}\,  U_0$ is independent of how $\tau$ as well as $U_0$  extends $f|_\Sigma$ (respectively~$\f|_\Sigma$) off~$\, \Sigma$ and thus determines a well-defined solution
operator
$$
\cO|_{h_0=d} : \big(\Gamma {\mathcal F}[1]|_\Sigma,\Gamma
\ce[1]|_\Sigma\big) \to \Gamma {\mathcal F}[1]\, .
$$
\end{corollary}

\begin{remark} \label{QPEcase} In the spirit of Remark ~\ref{PEcase}
  above some comments, extending the Corollary, are worthwhile. These
  again use Theorem ~\ref{logsol}.

First note that, in the solution, the coefficient of $\log \si$ is
\begin{equation}\label{logt1}
- \frac{1}{(n)!(n-1)!} P_{n} U_0
\end{equation}
where $P_n$ is the tangential operator from Theorem ~\ref{tanth}.  It
is natural to consider this in the case that $U_0= - \log \tilde\mu$,
where $\tilde\mu$ is some extension to $\Gamma\ce_+[1]$ of a section
$\mu\in \Gamma\ce_+[1]|_\Sigma$ so that $U_0\in {\mathcal F}[-1]$.
(Recall that for log densities pointwise multiplication is a bundle isomorphism, so 
here we are simply multiplying the entire solution by a factor $-1$.)

Then specialising to the case of $(M,c,\si)$ Poincar\'e-Einstein we
have $I^2=1$, and from~\nn{FGrQform} we see that in the interior scale
the problem being solved is 
\begin{equation}\label{PElogform}
-\Delta^{\g} \, U = n, 
\end{equation}
In this case, from Theorem ~\ref{Qthm}, we get
\begin{itemize}

\item If $n$ is even then the log coefficient~\nn{logt1} is 
$$
\frac{\scalebox{1.1}{$ Q^{g^\mu_\Sigma}$}\ }{2^{n-1}\big(\frac n2-1\big)!\, \big(\frac n2\big)!} 
$$
where $Q^{g^\mu_\Sigma}$ is the Q-curvature of $g^\mu_\Sigma$.

\vspace{1mm}
\item If $n$ is odd then the log term vanishes, and the solution is smooth.

\end{itemize}
In fact~\nn{PElogform} is precisely the problem studied by
\cite{FGrQ} and~\cite{FH} (noting the different sign of Laplacian these
sources use)  and these results are consistent with their findings.
\end{remark}

\subsection{The formal solutions in the interior scale} \label{total}
Summarising the above, away from the weights $w_0= \{
\frac{j-d}{2}\mid j\in \mathbb{Z}_{\geq 1}\}$, we found two
independent formal solutions to the equation $I\cdot D f=0$ along
$\Sigma$. These take the form
$$
F=f_0+\si f_1 +\cdots \  \mbox{ and } \ \ \sigma^{h_0-1}G = \sigma^{h_0-1} (\f_0+\si \f_1+\cdots )\, ,
$$
each of which is a density valued section of $\ce^\Phi[w_0]$ with $w_0\notin
\{\frac{j-d}{2}\mid j\in \mathbb{Z}_{\geq 1}\}$. 

We may eliminate the weight by choosing a scale $\tau\in \Gamma\ce[1]$
which extends to the boundary $\Sigma$, so equivalently a metric
$\overline{g}=\tau^{-2}\bg \in c$.
Then the solutions, in terms of the scale $\overline{g}$, are
$F^{\overline{g}}= \tau^{-w_0} F$ and $G^{\overline{g}}= \tau^{-w_0}
G$ in $\ce^\Phi=\ce^\Phi[0]$.  Note that $r:=\si/\tau$ is the function
$r$ which gives the density $\si$ in the scale $\tau$. Thus working in
the metric $\overline{g}$ the solutions take the form
$$
 F^{\overline{g}}=f_0+ r f_1 +\cdots \  \mbox{ and }\ \  \big[\sigma^{h_{0}-1}G\big]^{\overline{g}} = r^{h_0-1}(\f_0+ r \f_1+\cdots )
$$
{\it where now the $f_i$ and $\f_i$ are of weight 0} (but we have
retained the earlier notation for simplicity). Explicit expressions for
the unweighted tractor fields $f_i$ and $\f_i$ follow from
Proposition~\ref{twosol}.

On the other hand, for the purpose of comparison with the existing
literature, it is interesting to understand the solutions in terms of
the canonical generalised scale $\si$.  Away from $\Sigma$, we have
$F^{\g}= \si^{-w_0} F$ or in terms of the scale $\overline{g}$,
$F^{\g}= r^{-w_0}F^{\overline{g}}$, and similarly for $\f^{\g}$ which
may then be expressed in terms of $G^{\overline{g}}$. Thus, in terms
of the canonical scale $\g$, a general formal solution takes the form
$$
f= r^{n-s}F^{\overline{g}}+ r^s G^{\overline{g}}
$$
where, recall, $s:=w_0+n$. This equation (specialized to conformal
densities rather than generic tractors), alongside the
expression~\nn{IdotDsc} for the $I\cdot D$ operator in the scale $\g$,
gives a precise dictionary between our methods and the scattering
problem stated in equation 3.2 of~\cite{GrZ}.

At weights $w_0\in\{\frac{j-d}{2}\mid j\in \mathbb{Z}_{\geq 1}\}$
where log solutions appear, for the scale $\overline g=\tau^{-2}\bg$
we take the density $\tau$ extending to the boundary to be the same as
that appearing in the log terms in~\nn{logf}. In that choice of scale
$\tau$ itself is represented by the constant function~$1$, thus a
similar analysis to above shows that in the scale $\g$ we have
solutions of the form
\begin{equation}\label{logint}
f= r^{n-s}F_{h_0-2}^{\overline{g}}+ r^s \log r\,   \F^{\overline{g}} +  r^{s+1} \H^{\overline{g}}  + r^s G^{\overline{g}} \, ,
\end{equation}
 where now $s\in\{\frac{j+n-1}{2}\mid
j\in \mathbb{Z}_{\geq 1}\}$, the field $ G^{\overline{g}}$ (the
``second solution'' part) is of the form as above, and the fields
$F_{h_0-2}^{\overline{g}}$, $\F^{\overline{g}}$ and
$\H^{\overline{g}}$ take the form
$$
 F_{h_0-2}^{\overline{g}} =f_0+ r f_1 +\ \cdots\ +r^{2s-n-1}f_{2s-n-1} \, ,\quad \F^{\overline{g}} = \f_0+ r \f_1+\ \cdots 
$$
and
$$ 
 \H^{\overline{g}}=b_0 + r b_1 + \ \cdots\, .
$$
These last three mentioned terms are really part of the single log
type solution and the details can be explicitly read off from
Theorem~\ref{logsol}.  In the special case $s=\frac{n}2$ (so that
$h_0=1$), the polynomial term $F_{h_0-2}^{\overline{g}}$ is
omitted. For comparison with the existing literature, note that the
non-log terms $r^{n-s}F_{h_0-2}^{\overline{g}}$, $r^{s+1}
\H^{\overline{g}}$, and $ r^s G^{\overline{g}} $ in~\nn{logint} could
be swept together into a single term of the form $r^{n-s}F$, with $F$
of the form $ f_0+ r f_1 +\ \cdots\ $, so we have
$$
f= r^{n-s}F +  r^s \log r\,   \F^{\overline{g}}, 
$$ 
but this hides how independent solutions combine to give the general solution.

Continuing in this weight range $w_0\in\{\frac{j-d}{2}\mid j\in
\mathbb{Z}_{\geq 1}\}$, for almost Einstein structures
({\it e.g.} Poincar\'e-Einstein, the case considered in~\cite{GrZ}), when
$h_0$ is even so $s\in\{\frac{j+n-1}{2}\mid j\in 2\mathbb{Z}_{\geq
  1}\}$) there are no log terms in the solution, as explained in
Remark~\ref{PEcase}. Indeed, in~\cite{de Haro,GrZ}, by differing methods, the authors found a log obstruction 
only for $2s-n$ even. Also, essentially, the same analysis
as above applied to the log density problem of Section~\ref{logdpr}
confirms the form of solution predicted for the equivalent interior
problem formulated in~\cite{FGrQ}.

Finally, to complete a circle of ideas, let us write out the operator
$I\cdot D$ in the scale $\tau$ such that the function~$r$ above is a
defining function in the sense of~\cite{GrL}. In that case the metric
$\g=\frac{\overline g}{r^2}$ where $\overline g$ extends to the
boundary and has normal form
$$\overline g = dr^2 + h_r\, .$$
Then, an easy application of~\nn{degLa} gives the expression for
$I\cdot D$ acting on scalar densities of weight~$w$ and we find
$$
I\cdot D\stackrel{\overline g}{=}-\big(r\partial_r
-2s+n+1\big)\partial_r - H_r \big(r\partial_r +n-s\big) -r
\Delta_{h_r}\, .
$$
with $H_r=\frac12h^{ij}\partial_r h_{ij}$ and $h_{ij}$ is the metric defined by $h_r$ at fixed $r$.
This is the operator employed in the proof of Proposition~4.2 of~\cite{GrZ}.

\section{The Fefferman--Graham ambient metric and 
almost Einstein spaces} \label{ambsect}

Recall from Section ~\ref{dlogd} that a conformal structure (of any
signature $(p,q)$) is equivalent to the ray bundle $\pi:\cG\to M^d$ of
conformally related metrics.  Let us use $\rho $ to denote the ${\Bbb
  R}_+$ action on $ \cG$ given by $\rho(t) (x,g_x)=(x,t^2g_x)$.  An
{\em ambient manifold\/}  is a smooth $(d+2)$-manifold $\aM$ endowed
with a free $\Bbb R_+$--action $\rho$ and an $\Bbb R_+$--equivariant
embedding $i:\cG\to\aM$.  We write $\X\in\Gamma(T \aM)$ for the
fundamental field generating the $\Bbb R_+$--action.  That is, for
$f\in C^\infty(\aM)$ and $ u\in \aM$, we have $\X
f(u)=(d/ds)f(\rho(e^s)u)|_{s=0}$.  For an ambient manifold $\aM$, an
{\em ambient metric\/} is a pseudo--Riemannian metric $\h$ of
signature $(p+1,q+1)$ on $\aM$ satisfying the conditions: (i) ${\Cal
L}_{\sX}\h=2\h$, where $\Cal L_{\sX}$ denotes the Lie derivative by
$\X$; (ii) for $u=(x,g_x)\in \cG$ and $\xi,\eta\in T_u\cG$, we have
$\h(i_*\xi,i_*\eta)=g_x(\pi_*\xi,\pi_*\eta)$. In~\cite{FGast} (and see
\cite{FGrnew}) Fefferman and Graham considered formally the Gursat
problem of obtaining $\Ric(\h)=0$. They proved that for the case of
$d=2$ and $d\geq 3$ odd this may be achieved to all orders, while for
$d\geq 4$ even, the problem is obstructed at finite order by a natural
2-tensor conformal invariant (this is the Bach tensor if $d=4$, and is
called the Fefferman--Graham obstruction tensor in higher even
dimensions); for $d$ even one may obtain $\Ric(\h)=0$ up to the
addition of terms vanishing to order $d/2-1$. (See~\cite{FGrnew} for
the statements concerning uniqueness. For extracting results via
tractors we do not need this, as discussed in {\it e.g.}\
\cite{CapGoamb,GoPetCMP}.) We shall henceforth call any (approximately
or otherwise) Ricci-flat  metric on $\aM$  a {\em Fefferman--Graham
  metric}.   In the subsequent
discussion of ambient metrics all results can be assumed to hold
formally to all orders unless stated otherwise.

In the following discussion we typically use bold symbols or tilded
symbols for the objects on $\tilde{M}$. For example
$\boldsymbol{\nabla}$ denotes the Levi-Civita connection on
$\tilde{M}$.  
Familiarity with the 
treatment of the Fefferman--Graham metric, as in {\it e.g.}~\cite{CapGoamb,GoPetobstrn} 
or~\cite{BrGodeRham}, will be assumed.
In particular, we
shall use that suitably homogeneous tensor fields of
$\tilde{M}|_\cG$ correspond to tractor fields. This correspondence is
compatible with the Levi-Civita connection in that each weight zero
tractor field $F$ on $M$ corresponds to (the restriction to $\cG$ of)
a homogenous tensor field $\aF$ (of the same tensor rank) on $\aM$
with the property that it is parallel in the vertical direction, that
is $\aNd_{\sX} \aF=\X^A \aNd_A \aF= 0$ along $\cG$.  The metric $\h$
and its Levi-Civita connection $\aNd$ on~$\aM$ determine a metric and
connection on the tractor bundles, and by~\cite[Theorem 2.5]{CapGoamb}
this agrees with the normal tractor metric and connection. 
We use abstract indices in an obvious way on
$\aM$ and these are lowered and raised using $\h_{AB}$ and its inverse
$\h^{AB}$.

We shall say $\aF$ is homogeneous of {\em weight} $w_0\in \mathbb{C}$
if $\aNd_{\sX} \aF=w_0 \aF$, and this corresponds to a tractor field
$F$ of weight $w_0$.  We shall always take such fields to be extended
off $\cG$ smoothly and so that $\aNd_{\sX} \aF= w_0 \aF$ on
$\aM$. Along $\cG$ then $\aNd_{\sX}$, as applied to tensor fields of
well defined weight, gives an ambient realisation of the weight
operator. Note also that if $\tau \in \Gamma \ce_+[w_0]$  and $\boldsymbol{\tau}$ is the
corresponding ambient function, positive and homogeneous of weight
$w_0$, then
$$
[\aNd_{\sX} ,\log \boldsymbol{\tau} ]= w_0,
$$
compatible with our extension of the weight operator to log densities.

In this picture the operator $ \aD_A= \big(d+2\, w_0 -2\big) \aNd_A -
\aX_A\, \al $ on tensors homogeneous of weight $w_0$ corresponds to
the tractor-$D$ operator as applied to tractors of weight $w_0$.  Thus
we equivalently view this as a restriction of
$$
\aD_A= \aNd_A(d+2\aNd_{\sX}-2) +\aX_A\, \bs{\Delta}.
$$
This acts tangentially along the submanifold $\cG$ in $\aM$,
\cite{BrGodeRham} and~\cite{GoPetobstrn}, and is compatible with our
treatment of log-densities.

Given a (generalised) scale $\si\in \Gamma\ce[1]$ on $M$, we shall write
$\asi$ for the corresponding homogeneous weight 1 function on $\cG$
and its homogeneous extension to $\aM$.  Then with 
$$
\I_A:=\frac{1}{d}\aD_A \asi
$$
a restriction of the
differential operator
\begin{equation}\label{IdotDamb}
\I\cdot \aD: =  \aNd_{\I} (d+2\aNd_{\sX}-2) - \tilde \sigma \al
\end{equation}
(on $\aM$)  lifts the operator 
 $I\cdot D$ on $M$, enabling calculations on $\aM$. An important example
of such  is to compute the commutator
\begin{eqnarray*}
[\I\cdot \aD,\log\tilde\si]&=&
2\aNd_{\I}+\frac{(\aNd_{\I}\tilde\si)}{\tilde \si}\big(d+2\aNd_{\sX}-2\big)\\[1mm]
&&-\tilde\si\big(\aNd_A\frac{(\aNd^A\tilde\si)}{\tilde \si}+\frac{(\aNd_A\tilde\si)}{\tilde \si}\aNd^A\big)\\[3mm]
&=&-\frac{2}{d}(\al\tilde\si) \aNd_{\sX}
+\frac{(\aNd^A\tilde\si)(\aNd_A\tilde\si)-\frac1d \tilde \si (\al\tilde\si)}{\tilde \si}\big(d+2\aNd_{\sX}-2\big)\\[1mm]
&&-(\al\tilde\si)+\frac{(\aNd^A\tilde\si)(\aNd_A\tilde\si)}{\tilde \si}\\[3mm]
&=&
\frac{(\aNd^A\tilde\si)(\aNd_A\tilde\si)-\frac2d \tilde \si (\al\tilde\si)}{\tilde \si}\big(d+2\aNd_{\sX}-1\big)\, .
\end{eqnarray*}
Here the first line used the identity $[\aNd_{\sX},\log\tilde\s]=1$
while the second equality relies on the expression for the ambient
scale tractor $\I_A=\aNd_{\I}\tilde\si-\frac1d \aX_A \al\tilde\si$.
From this we recognize that the last the line of the above equals
$\frac{\I^2}{\tilde\si}(d+2\aNd_{\sX}-1)$, along $\cG$, and have therefore,
remembering that $y=-\frac1{I^2}I\cdot D$, established the following result.
\begin{proposition}\label{toadd2}
Suppose that $(M,c,\sigma)$ is such that $I^2$ is nowhere vanishing.
Then on any section of a weighted tractor bundle
$$
[y,\log x]
=-x^{-1}\ (h-1)\, .
$$
\end{proposition}

\subsection{The almost Einstein setting and
  Theorem~\ref{tanein}} \label{aefg} Here we consider the case
$(M^d,c,\si)$ almost Einstein, so $I_A=\frac{1}{d}D_A\si$ is
parallel. Let us first consider the case of $d$ odd.  The homogeneous
degree 1 function $\asi$ on $\cG$, corresponding to $\si$, can be
extended off~$\cG$ smoothly to a homogeneous degree 1 function on
$\aM$ that also satisfies $\al \asi=0$ formally to all
orders~\cite{GJMS}. In this case, which we henceforth assume, $\I_A$
is parallel on~$\aM$~\cite{Goal}.  (Reference~\cite{Goal} treats the
case of Riemannian signature $(M,c)$, and $I^2=1$, there is little
difference in moving to general signature and constant $I^2\neq 0$.)
 
Let us assume then that the ambient manifold $\aM$ is equipped with a 
  parallel
vector $\I^A=\frac 1d \aD^A \tilde \si$, of non-zero length, and the zero locus
$\tilde{\Sigma}=\mathcal{Z}(\asi)$ of $\asi=\aX^A \I_A$ is
non-empty. Note that $\aNd \tilde{\sigma}$ is nowhere vanishing on
$\tilde{\Sigma}$, which is thus necessarily a hypersurface. By homogeneity,
$\tilde{\Sigma}\cap \cG=\pi^{-1} (\Sigma )$, where $\Sigma=
\mathcal{Z}(\si)$.

Let $\tilde{f}_0$ be a section of a tensor bundle on $\aM$ which is
homogeneous of weight $\frac{k-n}{2}$, where $n=d-1$. We wish to
consider $(\I\cdot\aD)^k \tilde f_0$ along $\tilde{\Sigma}$. By the
lift of the calculations from Section~\ref{tanS} we conclude that
$(\I\cdot\aD)^k$ acts tangentially on $\tilde{f}_0$, along
$\tilde{\Sigma}\cap \cG$. In fact it is easily verified that $\tilde
\si$, $-\frac{1}{\I^2}\I\cdot\aD $ and $(d+2\aNd_{\sX})$ generate an
$\frak{sl}(2)$ (compatible in the obvious way with $\frak{g}$) 
and that $(\I\cdot\aD)^k$ acts tangentially on
$\tilde{f}_0$, along $\tilde{\Sigma}$.
 Thus to simplify
calculations we extend smoothly $\tilde{f}_0$ smoothly off
$\tilde{\Sigma}$ by the equation $\I^A\aNd_A \tilde{f}_0=0$. 
  It is easily
verified that (for any weight $w_0$) this is consistent with the
homogeneity assumption $\aNd_{\sX} \tilde{f}_0=w_0 \tilde{f}_0$.

For $\tilde{f}_0$ of any weight $w_0$, and satisfying $\I^A\aNd_A
\tilde{f}_0=0$  we have 
\begin{eqnarray*}
(\I\cdot\aD)^k \tilde f_0&=&-(\I\cdot \aD)^{k-1} (\tilde\sigma \al \tilde f_0)
\\[2mm]&=&- \tilde \si (\I\cdot \aD)^{k-1}\tilde\al  f_0-\I^2(k-1)(d+2w_0-k-2) (\I\cdot \aD)^{k-2} \al \tilde f_0\, .
\end{eqnarray*}
On the other hand using that $\I$ is parallel, we have that
$[\I^A\aNd_A,\al]$~$=$~$0$, on tensor fields over $\aM$. These 
 observations yield the following useful identity.
\begin{lemma}\label{alem} 
\begin{eqnarray} 
\!(\I\cdot \aD)^{k-2j} \al^j \tilde f_0 \!&\!\!=\!\!&\!\! -\tilde \sigma (\I \cdot \aD)^{k-2j-1} \al^{j+1} \tilde f_0\nonumber\\[2mm]
&&\!\!-\I^2 (k-2j-1) (d+2w_0-k-2j-2) (\I.\aD)^{k-2j-2} \al^{j+1} \tilde f_0 ,\nonumber
\end{eqnarray}
where $w_0$ is the homogeneity weight of $\tilde f_0$.
\end{lemma}

From this we obtain the following result.
\begin{proposition}\label{GJMSprop}
 For $\tilde{f}_0$ any tensor field, of homogeneity weight
  $w_0=\frac{k-d+1}{2}$, along $\tilde{\Sigma}$ we have
$$
(\I\cdot \aD)^k\tilde f_0=
\left\{
\begin{array}{cc}
\big[(k-1)!!\big]^2
(\I^2 \al_{\tilde{\Sigma}})^{\frac k2}\tilde f_0 + { O}(\tilde \sigma)\, , & k \mbox{ even,}\\[2mm]
{ O}(\tilde \si)\, ,& k \mbox{ odd.}
\end{array}
\right. ~,
$$
where $\al_{\tilde{\Sigma}}$ is the induced intrinsic Laplacian of 
$\tilde{\Sigma}$ coupled to the Levi-Civita connection of $\aM$. 
\end{proposition}
\begin{proof}
Since $(\I\cdot \aD)^k$ acts tangentially there is no loss of generality in assuming that $\I^A\aNd_A \tilde{f}_0=0$. Then 
 from the Lemma and an obvious induction we obtain 
$$
(\I\cdot \aD)^k\tilde f_0=
\left\{
\begin{array}{cc}
\big[(k-1)!!\big]^2
(\I^2 \al)^{\frac k2}\tilde f_0 + { O}(\tilde \sigma)\, , & k \mbox{ even,}\\[2mm]
{ O}(\tilde \si)\, ,& k \mbox{ odd.}
\end{array}
\right.
$$
But using that along $\tilde{\Sigma}$ we have $\al=
\al_{\tilde{\Sigma}} + \frac{1}{\I^2}\I^A\I^B\aNd_A\aNd_B $ and using
again that $[\I^A\aNd_A,\al]$~$=$~$0$ the result follows immediately.
\end{proof}

We see that for $k$ odd, on an almost Einstein manifold $P_k$ is the
zero operator on $\Sigma$, as claimed in Proposition~\ref{zero}. The statement
there concerning the leading order term of the operator is also an easy
consequence of the calculation above.

Now for simplicity let us assume $I^2=1$.  
As established in
\cite[Theorem 6.3]{Goal} the induced Ricci metric
$\h_{\tilde{\Sigma}}$ on $\tilde{\Sigma}$ is Ricci flat and
$(\tilde{\Sigma},\h_{\tilde{\Sigma}})$ gives a Fefferman--Graham
ambient metric to all orders for the conformal structure
$(\Sigma,c_\Sigma)$. 

Now consider the case where $f_0$ is a density in $\ce[\frac{k-n}{2}]$
and $\tilde{f}_0$ the corresponding homogeneous function on $\aM$. On
the Fefferman--Graham space $(\tilde{\Sigma},\h_{\tilde{\Sigma}})$ the
powers $ \al_{\tilde{\Sigma}}^{\frac{k}{2}}$, for even $k$, are
tangential along $\cG \cap \tilde\Sigma$; for $k\leq n$ these give the
GJMS operators, by definition~\cite{GJMS}.  It follows at once that
for $f_0$ a density in $\ce[\frac{k-n}{2}]$ we have
$$
(-\frac{1}{\I^2}\I\cdot \aD)^k f_0
=(-1)^k\big[(k-1)!!\big]^2 \mathcal{P}_k \tilde f_0\, , \quad k\in \{2,4,\cdots ,n/2 \}
$$ 
where $\mathcal{P}_k$ is the order $k$ 
GJMS operator. This establishes Theorem~\ref{tanein} for $d\geq 3$ odd. 

For $d\geq 4$ even all the above goes through, except that the usual
construction as in~\cite{FGast}, for the Fefferman--Graham ambient
metric is obstructed at finite order, as stated earlier. So also is
the ``harmonic'' extension of $\asi$ off $\cG$, we may obtain for
example that $0=\al \asi=\al^2 \asi= \cdots = \al^{d/2} \asi$ along~$\cG$, but at
the next order the problem is potentially obstructed (and
$\al^{d/2+1}$ depends on $c$ beyond the order of the obstruction).
However by a straightforward counting of derivatives it is 
verified that we recover the usual GJMS operators to at least the
order claimed in Theorem~\ref{tanein}.

\end{document}